%% file: main.tex
\pdfoutput=1
\documentclass[review,onefignum,onetabnum]{siamart171218}

\usepackage[utf8,latin1]{inputenc}

\usepackage{hyperref}

\usepackage{float}
\usepackage{amssymb}
\usepackage{tikz-cd}
\usepackage{tikz}
\usepackage{pgfplots}

\usepackage[frozencache=true,cachedir=.]{minted} 

\usepackage{graphics}
\pgfplotsset{compat=newest}
\usepgflibrary{patterns,patterns.meta}
\usetikzlibrary{patterns,patterns.meta}
\usepgfplotslibrary{external}

\pgfplotsset{compat=1.18} 

\usepackage{todonotes}
\usepackage{xspace}
\usepackage{bm}
\usepackage{pgf}
\usetikzlibrary{arrows, calc, decorations.pathmorphing, shapes}

\newcommand{\ccc}[1]{{{\color{black}{#1}}}}

\newcommand{\lb}[1]{{\color{black} #1}}
\newcommand{\cc}[1]{{\color{black} #1}}
\newcommand{\mbs}[1]{{\color{black} #1}}

\pgfdeclarepatternformonly[12pt]{MyLines}{\pgfqpoint{-4pt}{4pt}}{\pgfqpoint{12pt}{12pt}}{\pgfqpoint{-4pt}{4pt}}%
{
\pgfsetcolor{\tikz@pattern@color}
\pgfsetlinewidth{3pt}
\pgfpathmoveto{\pgfqpoint{0pt}{0pt}}
\pgfpathlineto{\pgfqpoint{2.1pt}{2.1pt}}
\pgfusepath{fill}}

\input{ex_shared}

\ifpdf
\hypersetup{
  pdftitle={Multihomogeneous Castelnuovo-Mumford regularity and Gr\"obner bases},
  pdfauthor={Mat{\'i}as R. Bender, Laurent Bus{\'e}, Carles Checa, Elias Tsigaridas}
}
\fi

\DeclareMathOperator{\DRL}{DRL}
\DeclareMathOperator{\reg}{reg}
\DeclareMathOperator{\xreg}{x-reg}
\DeclareMathOperator{\yreg}{y-reg}

\DeclareMathOperator{\xtor}{\mathfrak{R}_x}
\DeclareMathOperator{\ytor}{\mathfrak{R}_y}

\DeclareMathOperator{\ini}{in}

\DeclareMathOperator{\bigin}{bigin}
\DeclareMathOperator{\multigin}{multigin}

\DeclareMathOperator{\supp}{Supp}

\DeclareMathOperator{\sat}{sat}

\DeclareMathOperator{\HF}{HF}

\DeclareMathOperator{\gin}{gin}

\DeclareMathOperator{\GL}{GL}
\DeclareMathOperator{\cd}{cd}

\DeclareMathOperator{\xsat}{{xsat}}

\DeclareMathOperator{\mx}{\mathfrak{m}_x}
\DeclareMathOperator{\my}{\mathfrak{m}_y}
\DeclareMathOperator{\mb}{\mathfrak{b}}

\definecolor{almond}{rgb}{0.94, 0.87, 0.8}
 \definecolor{aliceblue}{rgb}{0.94, 0.97, 1.0}

\definecolor{airforceblue}{rgb}{0.36, 0.54, 0.66}

\newcommand{\ideal}[1]{( #1 )\xspace}
\newcommand\underrel[2]{\mathrel{\mathop{#2}\limits_{#1}}}
  \def\zzz#1 #2\relax{%
\expandafter\def\csname pgfk@/tikz/external/system call\endcsname{etex '&mylatexformat' #2}%
}

\expandafter\expandafter\expandafter\zzz\csname pgfk@/tikz/external/system call\endcsname\relax
\begin{document}

\maketitle


\begin{abstract}
We study the relation between the bigraded Castelnuovo-Mumford regularity of a bihomogeneous ideal $I$ in the coordinate ring of the product of two projective spaces and the \lb{bi}degrees of a Gr\"obner basis of $I$ with respect to the degree reverse lexicographical monomial order in generic coordinates.
For the single-graded case, Bayer and Stillman unraveled all aspects of this relationship forty years ago and these results led to complexity estimates for computations with Gr\"obner bases. 
We build on this work to introduce a bounding region of the bidegrees of minimal generators of bihomogeneous Gr\"obner bases for $I$. We also use this region to certify the presence of some minimal generators close to its boundary.
Finally, we show that, up to a certain shift, this region is related to the bigraded Castelnuovo-Mumford regularity of $I$.

\end{abstract}

\maketitle





 






\section{Introduction}

Given a homogeneous polynomial ideal,
we can answer 
most, if not all, of the (algorithmic or mathematical)
problems related to it
by exploiting Gr\"obner bases computations~\cite{buchberger}.
Such problems include, but are not limited to,
testing the membership of a polynomial in the ideal,
performing boolean operations,
computing its degree and dimension,
eliminating variables and computing elimination ideals or computing the  minimal free resolution.
Consequently, developing algorithms to compute Gr\"obner bases \cite{buchberger,f4,f5} and obtaining precise complexity estimates of their running times, as well as precise bounds on the degrees of the polynomials they manipulate, are central problems in computational algebraic geometry.

\subsection{The classical case} Let $I$ be a homogeneous ideal in a standard $\mathbb{Z}$-graded polynomial ring.
The monomial ideal $\ini(I)$  is the
\emph{initial ideal} of $I$, that is 
$$\ini(I) = (\ini(f) \ : \ f \in I) , $$
where the \textit{initial term} $\ini(f)$ is the leading term of the polynomial $f$ with respect to a fixed monomial order.
Then, roughly speaking, a Gr\"obner basis is a set of generators of $I$, the initial terms of
which generate the initial ideal; \lb{see for instance \cite[Chapter 15]{eisenbud1995}}.

In both theory \cite{whatcanbecomputed} and applications, for example in cryptography \cite{FJ-hfe-03,CamGol-invariant-2020}
and discrete optimization~\cite{DHK-dopt-book-12}, 
a common complexity parameter of related Gr"obner bases computations is the maximal degree of the generators of $\ini(I)$; it serves as an effective measure on the number of arithmetic operations that algorithms used to compute Gr\"obner bases need to perform.
It is known, even before the 
systematic algorithmic study of problems in algebraic geometry, that, in the worst case, 
we might need to operate with polynomials having double exponential degree, with respect to the number of variables, to perform various operations involving polynomial ideals~\cite{Hermann2019,MAYR1982305,whatcanbecomputed}. 
Instead of always relying on these pessimistic worst case bounds, a more refined analysis relates the maximal degrees of the minimal generators of $I$, 
and consequently the complexity of computing a Gr\"obner basis, 
with a central invariant coming from commutative algebra: the Castelnuovo-Mumford regularity \cite{bayer_criterion_1987, mumford}. In this way, we obtain a more detailed picture of the required operations and the difficulty of the corresponding algebraic problem.

Given a positive integer $m \in \mathbb{N}$, the ideal $I$ is called $m$-regular if $m + j$ upper bounds the  degrees of its $j$-th syzygies and, therefore, the \lb{$j$-th} Betti numbers of $I$. The Castelnuovo-Mumford regularity $\reg(I)$ is the minimal $m$ such that $I$ is $m$-regular. 
From the commutative algebra point of view, this invariant reflects many interesting properties of $I$. For instance, Eisenbud and Goto \cite{eisenbudgoto} use it to bound the degrees at which the graded pieces of the local cohomology modules $H^i_{\mathfrak{m}}(I)$ vanish, where $\mathfrak{m}$ is the ideal generated by all the variables; we refer to \cite{24hours} for a description of local cohomology. 
In addition, the degrees $m \in \mathbb{N}$ at which $I$ is $m$-regular coincide with the degrees at which the truncated ideals $I_{\geq m}$  have a linear resolution, meaning that all the maps appearing in a minimal free resolution are linear maps; see Definition~\ref{truncatedideals}.

As the regularity of an ideal $I$ bounds the degrees of \lb{its} syzygies, which in turn lead to a free resolution of $I$, it must also bound the degrees of a minimal set of generators of $I$.
If we apply this observation to the initial ideal, using any monomial order, then we deduce that
\begin{equation}
\label{firstbound}
    \max \{\text{degrees of generators of a minimal Gr\"obner basis} \}\leq \reg(\ini(I)).
\end{equation}
However, we cannot a priori know when this bound is tight. Moreover, it relies on the regularity of the initial ideal instead of the regularity of the ideal itself, which encodes \lb{some} algebraic and geometric properties of $I$. A relation between the regularity of an ideal $I$ and its initial ideal arises by noticing the upper semi-continuous behaviour of local cohomology under flat families (see \cite[Theorem 12.8]{hartshorne}), that is
\begin{equation}
\label{equationgap}
\reg(I) \leq \reg(\ini(I)).
\end{equation}
The previous two inequalities do not allow us to deduce any relation between the maximal degree of an element in a Gr\"obner basis and the Castelnuovo-Mumford regularity of the ideal. However, Bayer and Stillman \cite{bayer_criterion_1987} proved that, under additional assumptions, Eq.~\eqref{firstbound} and Eq.~\eqref{equationgap} are equalities, that is,
\begin{equation}
\label{eq:homo-eq-dgGB-reg}
    \max \{\text{degs. of elements in a minimal Gr\"obner basis} \} = \reg(\ini(I)) = \reg(I).
\end{equation}

The equalities hold under two assumptions:
\begin{itemize}
    \item The monomial order is the \textit{degree reverse lexicographical} monomial order; see \cite[\S 15.2]{eisenbud1995}.
    \item The ideal is in generic coordinates, that is, its initial ideal equals the \emph{generic initial ideal} $\gin(I)$; see \cite{galligo74}.
\end{itemize} 

\medskip

\noindent About the first assumption, the degree reverse lexicographical order is a common choice for the computation of Gr\"obner bases as, in several cases, it provides better bounds for the degrees of the generators of the Gr\"obner basis \cite{lazard, TRINKS1978475}.
Interestingly, we cannot relax this assumption:
if we use another monomial order, instead of degree reverse lexicographical, then we can always find ideals where the inequalities  Eq.~\eqref{firstbound} and Eq.~\eqref{equationgap} are strict  \cite{converse}.
\lb{The second assumption relies on the} study of ideals in generic coordinates initiated \lb{by} Galligo~\cite{galligo74} who proved the existence of generic initial ideals.
When $I$ is not in generic coordinates, the degrees of the generators of $\ini(I)$ and $\gin(I)$ might be different, either bigger or smaller. However, in some cases, we can replace this genericity assumption on the coordinates by a more \cc{restrictive change} of coordinates~\cite{bermejogimenez, hashemi2012efficient}.


In other words, Bayer and Stillman proved, under the previous two assumptions, that
the Castelnuovo-Mumford regularity gives a meaningful estimate on the complexity of computing Gr\"obner bases. This relation also clarifies the connection between the double exponential bound (in the number of variables) for the regularity \cite{galligo74, giusti, caviglia_sbarra_2005} and the double exponential bounds for the complexity of many problems in computational algebra~\cite{Hermann2019}. For example, for the family of ideals provided by Mayr and Meyer~\cite{MAYR1982305, SWANSON2003137},
the double exponential bounds are nearly optimal
and also settle the worst-case complexity of computing  a Gr\"obner basis.
\medskip

\subsection{Bihomogeneous polynomial systems} Polynomials and polynomial systems coming from applications usually have \lb{certain structures}. Therefore, \lb{instead of relying on general purpose algorithms, it is preferable to take advantage of  such  structures to improve the computational complexity of algorithms, especially for Gr\"obner bases computations.} 
\lb{An} important case of interest is the \lb{multihomogeneous structure, i.e.~polynomial systems defined} \cc{in the coordinate ring of a product of projective spaces}. This case involves polynomials homogeneous in several blocks of variables
and instead of a single degree, a tuple of degrees is assigned to each polynomial, corresponding to the degrees with respect to each of block of variables. \lb{In this paper, we will focus on bihomogeneous polynomial systems, that is to say polynomials that are homogeneous in two blocks of variables, which we will denote by $x$ and $y$.}

\medskip

The extension of Castelnuovo-Mumford regularity to the corresponding multihomogeneous ideals has attracted the interest of many researchers in the last three decades. The various related results concern 
the (suitable) definition of regularity and its main  properties \cite{bruce,MGS-mgcmr-04}, its connection to multigraded local cohomology modules  \cite{botbol2012castelnuovo, chardinholanda},
its relation with Betti numbers and virtual resolutions \cite{aramideh2021computing, berkeshelmansmith}, 
the special properties of ideals defining points and curves \cite{Cobb_2024, hapoints},
bounds on (degree) regions that extend previous results from the classical single graded case \cite{bruce2022bounds, boundsmaclagansmith, salatmolto}, 
and, of course, the (efficient) computation of Gr\"obner bases \cite{Bender_2018,FAUGERE2011406}.



In \cite{MGS-mgcmr-04}, Maclagan and Smith \lb{introduced a generalization of the Castelnuovo-Mumford regularity to the case of toric varieties, which encapsulates the bihomogeneous setting we are considering in this paper. In this latter case,  
the Castelnuovo-Mumford regularity of a bihomogeneous ideal is} a region of bidegrees, also denoted by $\reg(I) \subset \mathbb{Z}^2$ (Definition~\ref{maclagansmithpreelim}), corresponding to the vanishing of certain local cohomology modules with respect to the irrelevant ideal $\mb$ (the intersection of the ideals generated by each block of variables). They proved that this definition preserves some of the most relevant geometric properties of the classical Castelnuovo-Mumford regularity in the single graded case, in particular, it provides a bounding region for the bidegrees of the generators of any bigraded ideal (\cite[Theorem 1.3]{MGS-mgcmr-04}).

\medskip

\lb{The extension of the properties of generic initial ideals from the homogeneous to the bihomogeneous (actually multihomogeneous\footnote{\cc{We notice that, as explained in \cite[Example 4.11]{MGS-mgcmr-04}, if the multigrading is not standard, generic initial ideals need not exist.}}) setting has also attracted interest and they are several existing results in this direction \cite{ aramova2000bigeneric, conca2016CS, roemer} (see also Section \ref{sec::1})}. However, the connection between the Castelnuovo-Mumford regularity of a bihomogeneous ideals and the bidegrees of the generators of generic initial ideals \lb{is still an open problem}. \lb{More specifically, it is a natural question to understand to which extent the Bayer and Stillman criterion established in the single graded setting can be extended to the bigraded setting. This question, which is the main motivation of this paper, amounts to} unravel the relation between the Castelnuovo-Mumford regularity of a bihomogeneous ideal (and other related invariants) and the bidegrees of the minimal generators of its \textit{bigeneric initial ideal} (see Definition \ref{bigeneric}). This translates into characterizing the relation between this regularity and 
the bidegrees of the minimal generators of the degree reverse lexicographical Gr\"obner basis, after a generic change of coordinates that preserves the bigraded structure. \lb{We notice that unlike in the single graded setting}, the bidegrees of the generators of the bigeneric initial ideal depend on the choice of the relative order of the variables of different bidegrees (see Example \ref{ex:2-diff-orders}).


\subsection{Previous works} 
Aramova, Crona, and De Negri \cite{aramova2000bigeneric} and R\"omer \cite{roemer} studied the relation between the regularity \lb{of a bihomogeneous ideal $I$} and the bidegrees of the generators of \lb{its bigeneric initial ideal, denoted by $\bigin(I)$}, assuming that the block of the $x$ variables is smaller than the block of the $y$ variables.
\lb{For that purpose,} they introduced a notion of regularity, which we denote by $\xtor(I)$. \lb{It is an integer which is} defined in terms of the Betti numbers of $I$ with respect to one of the groups of variables \lb{(see Definition~\ref{roemerdefiniiton} for more details)}.
%
\lb{In  \cite{aramova2000bigeneric}, 
Aramova, Crona, and De Negri proved that the maximal degree with respect to the variables $x$ (resp. $y$) of a minimal generator of $\bigin(I)$ is equal to $\xtor(\bigin(I))$ (resp. $\ytor(\bigin(I))$) (see \cite[Theorem 2.2]{aramova2000bigeneric})}. Using \lb{the above-mentioned assumptions} on the monomial order, in \cite{roemer} R\"omer \lb{proved} that
$$\xtor(I) = \xtor(\bigin(I)).$$
\lb{Thus}, the description \lb{with respect to} the $x$ block of variables depends solely on the Betti numbers of $I$ (see \cite[Proposition 4.2]{roemer}). However, the same result does not hold for $\ytor(I)$ \lb{without changing} the monomial order. \lb{We emphasize that,} as $\xtor(I)$ is a single positive integer, it provides information on the maximal degree of a minimal generator of $\bigin(I)$ with respect to only one block of variables.  
\mbs{ Example~\ref{exampleintro} illustrates how $\xtor(I)$ (in purple) bounds the bidegrees of the minimal generators of $\bigin(I)$ and how loose this bounding region can be.} 

\subsection{Contributions}
\lb{This paper deals with the problem of determining regions in $\mathbb{Z}^2$ containing all bidegrees that are involved in the  
computation of a Gr\"obner basis of a bihomogeneous ideal $I$. Such a region is expected to be described in terms of algebraic invariants of $I$, such as local cohomology modules or Betti numbers.}

\lb{As a first contribution, we show that the Castelnuovo-Mumford of $I$ or $\bigin(I)$, which are the natural candidates to construct such regions, does not have the expected properties as in the single graded setting. For that purpose,     
we focus on the specific setting of ideals defining empty subschemes of $\mathbb{P}^n \times \mathbb{P}^m$ for which 
we establish a natural link between $\reg(I)$ and $\bigin(I)$ (Corollary~\ref{emptyimplication}).}
\mbs{This special case allows us to demonstrate how much more complicated the bigraded case can be than the single graded one. Concretely, in Section~\ref{sec:first_examples}}
we illustrate with examples that there may be bidegrees outside $\reg(I)$ for which there are no elements of the Gr\"obner basis of strictly bigger bidegrees. Moreover, there may be bidegrees in $\reg(I)$ such that there are generators of $\bigin(I)$ of strictly higher bidegrees; see Example~\ref{examplemat2} and Example~\ref{example3}. Furthermore, Example \ref{example3} illustrates that, in contrast to the classical case, $\reg(I)$ is not preserved by considering $\bigin(I)$, meaning that, in general, it is not enough to study $\reg(I)$ in order to understand $\reg(\bigin(I))$; see Remark \ref{rem:reg_and_bigin}.


\lb{The second and main contribution of this paper is the definition of a novel  region attached to a bihomogeneous ideal $I$, which we will denote by $\xreg(I)$. This region allows us to obtain novel bounding regions for the bidegrees of the generators of the $\bigin(I)$. Inspired by the works of Botbol, Chardin, and Holanda~\cite{botbol2012castelnuovo, chardinholanda}, it is defined by means of the vanishing of some graded components of local cohomology modules with respect to one block of variables.}


\begin{definition}
\label{def:partial-regularity-intro}
Let $I \subset \ccc{\mathbf{k}[x_0,\dots,x_{n},y_0,\dots,y_{m}]}$ be a bihomogeneous ideal. 
The \textit{partial regularity region}, $\xreg(I)$,
is a region of bidegrees $(a,b) \in \mathbb{Z}^2$, such that for all $i \geq 1$ and $(a',b') \geq (a - i + 1, b)$:
$$ H_{\mx}^i(I)_{(a',b')} = 0 , $$ 
where $\mx$ is the ideal generated by the $x$ block of variables.
\end{definition}
In Theorem~\ref{theoremcriterion}, we show that this region can be characterized in terms of a criterion that can be seen as an extension of Bayer and Stillman's criterion \cite[Theorem 1.10]{bayer_criterion_1987}. Then, \lb{assuming that the monomial order is a degree reverse lexicographical order} with the $x$ variables being smaller than the $y$ variables, 
we prove that (see \ccc{Theorem \ref{main-theo-sec-4}})
$$\xreg(I) = \xreg(\bigin(I)).$$ 
\lb{Moreover}, we show that we can use $\xreg(I)$ to deduce that there are no generators of $\bigin(I)$ of certain bidegrees (Theorem~\ref{theoremxreg}). \lb{Actually}, $\xreg(I)$ allows us to identify some regions where there must exist generators of $\bigin(I)$ (Theorem~\ref{theoremexistenceelementsxreg}). Finally, \lb{relying on} results of Chardin and Holanda~\cite{chardinholanda}, we provide links between $\xreg(I)$, the bigraded Castelnuovo-Mumford regularity (Theorem~\ref{bigtheoremregularity}) and Betti numbers (Proposition~\ref{theorembetti}). 
This implies, that, up to a shift in the degrees of the $x$ variables, we can also use $\reg(I)$ to provide a bounding region to the bidegrees of the generators of $\bigin(I)$.
\medskip

To summarize, our work is a step towards finding effective sharp bounding regions of the bidegrees of the polynomials \lb{appearing during} the computation of a bigraded Gr\"obner basis (and the study of the complexity of the related algorithms). These regions are obtained by means of the partial regularity $\xreg(I)$ which, up to a fixed order of the blocks of variables, is invariant \lb{when passing to bigeneric initial ideals}. However, there may be unbounded regions of bidegrees not intersecting $\xreg(I)$. These regions consist of bidegrees whose first component can only attain a finite number of values (see the region in blue in Figure \ref{fig:examplebruceroemerintro}).

The following example
illustrates our contribution.

\begin{example}
\label{exampleintro}
    This example is an adaptation of \cite[Example 4.3]{bruce} (similarly \cite[Example 1.4]{berkeshelmansmith}) and corresponds to a smooth hyperelliptic curve of genus $8$ embedded in $\mathbb{P}^2 \times \mathbb{P}^1$.  Consider the standard $\mathbb{Z}^2$-graded ring $\mathbb{C}[x_0,x_1,x_2,y_0,y_1]$ and the ideal
$$J = (y_0^2x_0^2 + y_1^2x_1^2 + y_0y_1x_2^2, y_0^3x_2 + y_1^3(x_0 + x_1)).$$ 
Let $I = J^{\sat} = (J:\mathfrak{b}^{\infty})$ be the saturation of $J$ with respect to the irrelevant ideal $\mathfrak{b} = (x_0y_0, x_0y_1, x_1y_0, x_1y_1, x_2y_0, x_2y_1)$ of $\mathbb{P}^2 \times \mathbb{P}^1$. In Figure \ref{fig:examplebruceroemerintro}, we have drawn the degrees of the generators of $I$ and of $\bigin(I)$, assuming a choice of monomial order such that the block of $x$ variables is lower than the block of $y$ variables. The integer $\mathfrak{R}_x(I)$ yields a tight bound for the degrees of the generators of $\bigin(I)$ with respect to the degree of the $x$'s. The partial regularity region $\xreg(I)$ provides a finer description of the bidegrees of the generators of $\bigin(I)$.
\begin{figure}[h]
\centering
\input{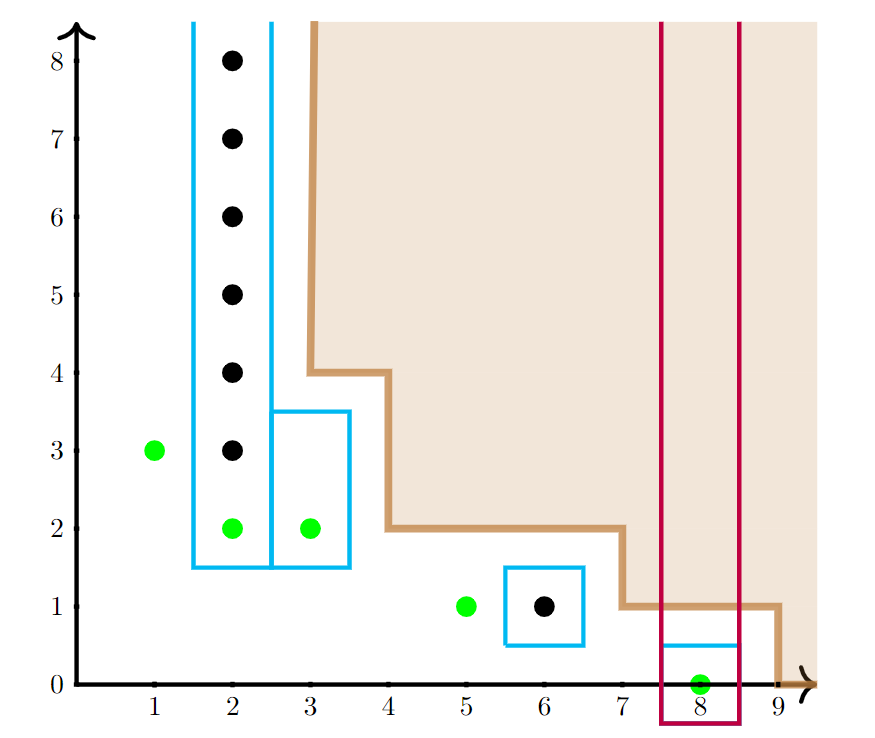}
        \caption{The \textcolor{green}{green dots} represent the degrees of the generators of $I$ and the black dots represent the degrees of the generators of the bigeneric initial ideals. In this example $\xtor(I) = 8$ (in \textcolor{purple}{purple}). In the region (in \textcolor{brown}{brown}), which \ccc{corresponds to a shift of $\xreg(I)$},  we can certify that there are no minimal generators of $\bigin(I)$ of those bidegrees. Moreover, $\xreg(I)$ can be used to certify the presence of generators of $\bigin(I)$ of the bidegrees marked in \textcolor{blue}{blue}.}
        \label{fig:examplebruceroemerintro}
    \end{figure}
\end{example}




\subsection{Content of the paper} 
In Section \ref{sec::1}, we present all the preliminary definitions and results on Gr\"obner bases, bigeneric initial ideals, local cohomology and multigraded Castelnuovo-Mumford regularity that are needed in this work. In Section \ref{sec:first_examples}, we illustrate the relation between $\reg(I)$ and the generators of $\bigin(I)$ with examples, \lb{focussing on the particular setting of ideals defining empty varieties}. In Section \ref{sec:partial-BS}, we introduce \lb{the partial regularity regions} $\xreg(I)$ and prove a criterion that \lb{can be seen as an extension of} the Bayer and Stillman criterion to the setting of bigraded ideals. In Section \ref{sec::5}, we derive the consequences of this criterion and use it to provide a description of the bidegrees of a minimal set of generators of $\bigin(I)$, as well as its relation with $\reg(I)$ and the Betti numbers. Finally, in Section \ref{sec:final}, we discuss open problems of our work, such as the tightness of the bounding regions and we discuss the multihomogeneous case. 





 
\medskip
 
\paragraph{Acknowledgments} This project has received funding from the European Union's Horizon 2020 research and innovation programme  under the Marie Sk\l{}odowska-Curie grant agreement N. 860843 and a public grant from the
Fondation Math\'ematique Jacques Hadamard. We are very grateful to Marc Chardin for interesting discussions and to the anonymous reviewers whose comments contributed to improve this paper.

\section{Preliminaries} 
\label{sec::1} 

We present the notation, definitions and results that we will need in the sequel.

\paragraph{Notation} Throughout the paper, we will use the following notation:

\begin{itemize}
    \item[-] Given two pairs $(a,b), (a',b') \in \mathbb{Z}^2$, we will write $(a,b) \geq (a',b')$ if $a \geq a'$ and $b \geq b'$.
    \item[-] We will write $(a,b) \gneq (a',b')$  if $(a,b) \geq (a',b')$ and $(a,b) \neq (a',b')$.
\end{itemize}


\subsection{Bihomogeneous ideals and bigeneric initial ideals}

Let $\mathbf{k}$ be a field of characteristic $0$. Let $S = \mathbf{k}[x_0,\dots,x_{n},y_0,\dots,y_{m}]$ be a ring with a \cc{standard} $\mathbb{Z}^2$-grading, such that $\deg(x_i) = (1,0)$ and $\deg(y_j) = (0,1)$. We write the monomials in $S$ as $x^{\alpha}y^{\beta} = x_0^{\alpha_0}\cdots x_n^{\alpha_n}y_0^{\beta_0}\cdots y_m^{\beta_m}$ for a vector $(\alpha,\beta) \in \mathbb{Z}^{n + 1} \times \mathbb{Z}^{m+1}$. A monomial $x^{\alpha}y^{\beta}$ has bidegree $(a,b)$ if $\sum_{i = 0}^n\alpha_i = a$ and $\sum_{j = 0}^m\beta_j = b$.

Let $\mx$ (resp. $\my$) be the ideal generated by the $x$ (resp. $y$) variables. We denote by 
$\mathbb{P}^{n} \times \mathbb{P}^{m}$ the biprojective space whose homogeneous coordinate ring is $S$ and by $\mathfrak{b} = \mx\my$ the irrelevant ideal. The main algebraic objects that we  manipulate are bihomogeneous polynomials and bihomogeneous ideals.
\begin{definition}
\label{def:bihomo-ideal}
    A polynomial $f = \sum_{\alpha,\beta}c_{\alpha,\beta}x^{\alpha}y^{\beta} \in S$ is bihomogeneous of bidegree $(a,b) \in \mathbb{Z}^2$ if all of its terms are monomials of bidegree $(a,b)$. An ideal $I \subset S$ is bihomogeneous if it can be generated by bihomogeneous polynomials. The bigraded part of bidegree $(a,b)$ of $I$ is the $\mathbf{k}$-vector space generated by all the polynomials of bidegree $(a,b)$ in $I$; it is denoted by $I_{(a,b)}$.
\end{definition}

In the rest of the paper, \textit{a linear $x$-form} is an element in $S_{(1,0)}$.

\begin{definition}
\label{def:DRL}
    Consider a degree reverse lexicographical monomial order $<$ (or $\DRL$) such that:
\begin{equation}
\label{eqmono}
    x_0 < \dots < x_n < y_0 \dots < y_m.
\end{equation}
 For two monomials $x^{\alpha}y^{\beta}$ and $ x^{\alpha'}y^{\beta'}$ of the same bidegree $(a,b) \in \mathbb{Z}^2$, the degree reverse lexicographical order satisfies the property that
$$ x^{\alpha}y^{\beta} < x^{\alpha'}y^{\beta'}  \iff \text{ the leftmost non-zero entry of } (\alpha'-\alpha, \beta' - \beta)\text{ is negative.}$$
Here, the leftmost non-zero entry refers to the vector $$(\alpha,\beta) = (\alpha_0,\dots,\alpha_n,\beta_0,\dots,\beta_m)$$ from left to right. 
\end{definition}

\begin{remark}
    We notice that there are many different
    degree reverse lexicographical monomial orders 
    corresponding to the various permutations of the variables in Eq.~\eqref{eqmono}.
    In our case, it is important to consider
    the variables in $x$ and the variables in $y$ and the relative order of the two blocks.
    This is so, because we will consider linear change of coordinates with respect to the two blocks and as a consequence, only the relative order of the two blocks matter for the structure of the generators of the monomial ideals that we consider; see, for instance, Example~\ref{ex:2-diff-orders}.
\end{remark}

One of the main properties of the degree reverse lexicographical monomial order is that if $x_0$ divides a monomial $x^{\alpha}y^{\beta}$, then it also divides every monomial $x^{\alpha'}y^{\beta'} < x^{\alpha}y^{\beta}$. In particular,
\begin{equation}
\label{drl}
    x_0\text{ divides }\ini(f) \implies x_0 \text{ divides }f.
\end{equation}

In what follows, except if we explicitly state otherwise, we will only use the $\DRL$ monomial order in Eq.~\eqref{eqmono}.

\begin{definition}
\label{def:ini-f}
Let $<$ be the degree reverse lexicographical monomial order defined previously.
Consider $f = \sum_{\alpha,\beta}c_{\alpha,\beta}x^{\alpha}y^{\beta} \in S$.
The \emph{initial monomial} (also known as leading monomial) of $f$
is the largest monomial with non-zero coefficient appearing in $f$
with respect to the monomial order $<$;
we denote it by $\ini(f)$.

Given an ideal $I \subseteq S$, we define its \emph{initial ideal}, $\ini(I)$, with respect to the monomial order $<$, as the ideal generated by the initial monomials of the polynomials in $I$; that is $\ini(I) = \ideal{ \ini(f) \,|\, f \in I }$.
\end{definition}

The definition of Gr\"obner bases for bihomogeneous ideals is as follows.

\begin{definition}
\label{def:bihomo-GB}
    A (bihomogeneous) \emph{Gr\"obner basis}, $G$, of a bihomogeneous ideal $I \subset S$ is a set of bihomogeneous polynomials  such that ${\ini(I)} = \ideal{  \ini(g) \,|\, g \in G }$.
    
    The Gr\"obner basis is minimal if the initial monomials of
    any proper subset of $G$ do not generate $\ini(I)$.
\end{definition}


Let $u \in \GL(n + 1) \times \GL(m + 1)$ be a block-diagonal matrix with entries in $\mathbf{k}$ and nonzero determinant and denote by $u^x$ and $u^y$ its two canonical blocks. This matrix defines a linear change of coordinates in $S$ as:
\begin{equation}
    u = (u^x,u^y): S \xrightarrow[]{} S \quad x_i \xrightarrow[]{} u^x_{i0}x_0 + \dots + u^x_{in}x_n \quad y_j \xrightarrow[]{} u^y_{j0}y_0 + \dots + u^y_{jm}y_m.
\end{equation}

For each polynomial $f \in S$, we define the polynomial $u \circ f$ as $f(u(x,y))$, which has the same bidegree as $f$. For any bihomogeneous ideal $I \subset S$ and any $u \in \GL(n + 1) \times \GL(m + 1)$, we  define the ideal $u \circ I = \ideal{ u \circ f \, |\, f \in I }$.

\begin{lemma}[Bigeneric initial ideal {\cite[Section~1]{aramova2000bigeneric}}]\label{bigeneric}
     For every \ccc{bihomogeneous ideal} $I \subset S$, there exists a Zariski open subset $U \subset \GL(n+1) \times \GL(m+1)$ 
    and a monomial ideal $\bigin(I)$ such that for any $u \in U$, $\ini(u \circ I) = \bigin(I)$. We call $\bigin(I)$ the bigeneric initial ideal of $I$. 
\end{lemma}

Bigeneric initial ideals generalize the generic initial ideals of homogeneous polynomials to the case of bihomogeneous (and eventually multihomogeneous) polynomials. Generic initial ideals have many useful properties; see \cite{Green1998}. In particular, the properties that we need are the following.

\begin{lemma}[{{\cite[Section~1]{aramova2000bigeneric}}}]
\label{lemmabiborel}
    For a bigeneric initial ideal $\bigin(I)$ the following hold:
    \begin{itemize}
        \item[-] If $x_ix^{\alpha}y^{\beta} \in \bigin(I)$, then $x_jx^{\alpha}y^{\beta} \in \bigin(I)$ for all $j \in \{i,\dots,n\}$.
        \item[-] If $y_ix^{\alpha}y^{\beta} \in \bigin(I)$, then $y_jx^{\alpha}y^{\beta} \in \bigin(I)$ for all $j \in \{i,\dots,m\}$.
    \end{itemize}
\end{lemma}

\begin{remark}
    For any monomial ideal, the above conditions appear under the name of bi-Borel fixed property. If we perform the linear change of coordinates with respect to only one block of variables, then we can recover a monomial ideal with the property of the lemma with respect to this block.
    This statement follows from the proof of the single graded case; see \cite[\dag 15.9]{eisenbud1995}.
\end{remark}

\begin{example}
\label{ex:2-diff-orders}
We consider Example \ref{exampleintro} in Section \ref{sec::1}.
In Figure \ref{fig:differentorder}, we can see the bidegrees of $\bigin(I)$, when we use different relative orders for the two blocks of variables. The left part of Figure \ref{fig:differentorder} depicts the degrees of the generators of $\bigin(I)$ using the monomial order in Eq.~\eqref{eqmono}, where the $x$ variables are the smallest. The right part of the figure shows the degrees of the generators of the bigeneric initial ideal if we reverse the order of the two blocks of variables; where the $y$ variables are the smallest.
    \begin{figure}
        \centering
        \input{tikzfiles/fig1-2.tikz}\input{tikzfiles/fig1-1.tikz}
        \caption{The \textcolor{green}{green dots} represent the degrees of the generators of $I$ and the black dots represent the degrees of the generators of the bigeneric initial ideals.}
        \label{fig:differentorder}
    \end{figure}
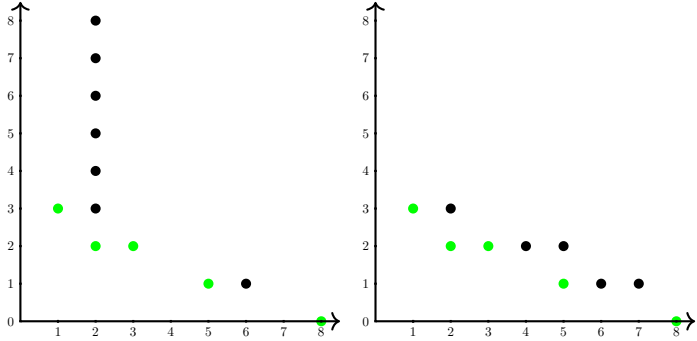
\end{example}


\subsection{Local cohomology modules and supports}

 In what follows, we review some properties of local cohomology and their sets of supports, which are relevant for studying the multigraded Castelnuovo-Mumford regularity.

\begin{notation}
    $H_{J}^i(I)$ will denote the $i$-th local cohomology module of a bihomogeneous ideal $I \subset S$ with respect to another ideal $J$. For the definition and properties of local cohomology, we refer to the books \cite{broadmann, 24hours}.
\end{notation}

\begin{definition}
    The \emph{cohomological dimension} of $I$ with respect to $J$ is:
\begin{equation}\label{cohomological} \cd_{J}(I) = \max ( \{0\} \cup \{i \in \mathbb{Z}_{>0} \text{ s.t. } H^i_{J}(I) \neq 0\} ).\end{equation}
\end{definition}


\begin{remark}
\label{remarkcohomological}
    The cohomological dimension is bounded above by the minimal number of generators of $J$; see \cite{broadmann}. In the case where $J$ is $\mx$ or $\my$, the minimal number of generators is $n + 1$ and $m + 1$, respectively. 
\end{remark}

The recent work of Chardin and Holanda \cite{chardinholanda}, see also \cite{botbol2012castelnuovo}, 
relates the vanishing of the local cohomology with respect to $\mathfrak{b}$ with the vanishing with respect to the ideals $\mx$ and $\my$. To state this result, we first need to introduce some additional notation.

\begin{definition}
\label{def:support}
Let $E \subset \mathbb{Z}^2$ be a subset. The subset $E^{\star}$ is defined as:
$$E^{\star} = \{(a,b) \in \mathbb{Z}^2 \,:\, \text{there exists } (a',b') \in E \text{ s.t. } (a',b') \geq (a,b) \}.$$

\end{definition}

In particular, we are interested in the case where the subsets $E$ are the supports of the local cohomology modules with respect to a homogeneous ideal $J$; in our case $J$ is $\mb, \mx$ or $\my$.

\begin{definition}
    Let $I,J \subset S$ be bihomogeneous ideals. The support of the local cohomology modules of $I$  with respect to $J$ are the bidegrees $(a,b) \in \mathbb{Z}^2$ such that there exists $i \geq 1$ for which $H^i_J(I)_{(a,b)}$ is not zero, i.e.,
\begin{equation}\label{supports}
    \supp_{\mathbb{Z}^2}(H^{\bullet}_{J}(I)) =
    \{(a,b) \in \mathbb{Z}^2 \,:\,  \textrm{there exists } i\geq 1 \textrm{ s.t. } H_{J}^i(I)_{(a,b)} \neq 0  \}.
\end{equation}
\end{definition}

The following theorem relates the supports of the local cohomology modules with respect to $\mathfrak{b}$ with the supports of the local cohomology modules with respect to $\mx$ and $\my$. 

\begin{theorem} 
[{\cite[Theorem 3.11]{chardinholanda}}]
\label{def:chardinholanda}
Let $I \subset S$ be a bihomogeneous ideal, then: 
    $$\supp_{\mathbb{Z}^2}(H^{\bullet}_{\mathfrak{b}}(I))^{\star} = \supp_{\mathbb{Z}^2}(H^{\bullet}_{\mx}(I))^{\star} \cup \supp_{\mathbb{Z}^2}(H^{\bullet}_{\my}(I))^{\star}.$$
\end{theorem}

\begin{remark}
\label{rem:lcohomo-when-I-eq-S}
If $I = S$, \ccc{the supports of the local cohomology with respect to $\mx$ (similarly for $\my$) are:
\begin{equation}
    \label{localcohomologyx}
\supp_{\mathbb{Z}^2}(H^{\bullet}_{\mx}(S)) = (-n-1,0) + (-\mathbb{N} \times \mathbb{N}),
\end{equation}
see \cite[Example 2.3]{chardinholanda}.}

\end{remark}



\begin{definition}
\label{truncatedideals}
    For any bihomogeneous ideal $I \subset S$ and $(a,b) \in \mathbb{Z}^2$, consider the truncated ideal 
    \[ I_{\geq (a,b)} := \bigoplus_{(a',b') \geq (a,b)}I_{(a',b')} . \]
\end{definition}

The following result from \cite{chardinholanda} relates the local cohomology with respect to $\mx$ of the truncated modules $I_{\geq (a,b)}$ with the local cohomology of $I$.

\begin{lemma} [{\cite[Proposition 4.4]{chardinholanda}}]
\label{lemmachardinholandasupports}
    Let $I \subset S$ be a bihomogeneous ideal and $(a,b) \in \mathbb{Z}^2$, then: 
    \begin{itemize}
        \item[i)] If $(a',b') \geq (a,b)$, then:
        $$ H^1_{\mx}(I_{\geq (a,b)})_{(a',b')} = H^1_{\mx}(I)_{(a',b')} .$$
        \item[ii)] For all $i \geq 2$, then:
        $$ H^i_{\mx}(I_{\geq (a,b)}) = H^i_{\mx}(I).$$
    \end{itemize}
\end{lemma}

\subsection{The bigraded Castelnuovo-Mumford regularity and the Betti numbers}

The bigraded generalization of the Castelnuovo-Mumford regularity, introduced by Maclagan and Smith \cite{MGS-mgcmr-04}, is one of the central algebraic objects in our study. It is defined by considering the vanishing of the local cohomology modules with respect to the irrelevant ideal $\mathfrak{b}$.
\begin{definition}
[{\cite[Definition 1.1]{MGS-mgcmr-04}}] 
\label{maclagansmithpreelim} Consider a bihomogeneous ideal $I \subset S$. The bigraded Castelnuovo-Mumford regularity $\reg(I)$ is the subset of $\mathbb{Z}^2$ 
containing bidegrees $(a,b)$ 
such that, for all $i \geq 1$
and for all $(a',b') \geq (a - \lambda_x,b - \lambda_y)$,
it holds
$$ H_{\mathfrak{b}}^i(I)_{(a',b')} = 0,$$
where  $\lambda_x + \lambda_y = i - 1$, with $\lambda_x, \lambda_y \in \mathbb{Z}_{\geq 0}$.
\end{definition}


The above definition preserves some of the classical properties of the Castelnuovo-Mumford regularity; for example, it \ccc{provides a bounding region} the degrees of the polynomial equations that cut out the variety defined by $I$. Moreover, Bruce, Cranton-Heller and Sayrafi proved that $(a,b) \in \reg(I)$ if and only if the truncated ideal $I_{\geq (a,b)}$ has a \textit{quasi-linear} resolution; see \cite[Theorem A]{bruce} for more details. Another important feature of the bigraded Castelnuovo-Mumford regularity is its relation with the generators of $I$.
\begin{theorem}[{\cite[Theorem 1.3]{MGS-mgcmr-04}}]\label{maclagansmithgenerators}
    Let $I \subset S$ be a \lb{bi}homogeneous ideal. If $(a,b) \in \reg(I)$ then, for all $(a',b') \gneq (a,b)$ there are no minimal generators of $I$ of degree $(a',b')$.
\end{theorem}

\begin{remark}
\label{rem:reg_and_bigin}
    Note that Theorem \ref{maclagansmithgenerators} implies that $\reg(\bigin(I))$ provides a \ccc{bounding region} for the bidegrees generators of $\bigin(I)$ \ccc{i.e.~a subset to the bidegrees that have no minimal generator of $\bigin(I)$ of higher degrees}. Moreover, along the same lines as in \cite[Proposition 3.16]{MGS-mgcmr-04}, we can prove that the following inclusion holds:\begin{equation}
    \label{inclusion}
        \reg(\bigin(I)) \subset \reg(I).
    \end{equation} However, as we will see in Example \ref{example3}, the previous two regularity regions will, in general, differ. \ccc{Therefore, unlike in the classical case, it is not enough to study $\reg(I)$ in order to understand $\reg(\bigin(I))$}.
\end{remark}

In the single graded case, the Castelnuovo-Mumford regularity can also be determined from Betti numbers, so the relation between the \lb{bi}graded Betti numbers and \ccc{Castelnuovo}-Mumford regularity is a natural question.

\begin{definition}
    \ccc{Let $F_{\bullet}$ be the minimal free resolution of $I$ with:}
    $$F_i = \bigoplus_{(a,b) \in \mathbb{Z}^2} S(-a, -b)^{\beta_{i,(a,b)}}.$$
Here, $S(-a,-b)$ denotes a shift in the grading, namely $S(-a,-b)_{(a',b')} = S_{(a'-a,b'-b)}$ for any $(a',b')$ and $(a,b)$ in $\mathbb{Z}^2$. \ccc{The $i$-th \emph{bigraded Betti numbers} of $I$} are $\beta_i(I) := \{(a,b) \in \mathbb{Z}^2 \,:\, \beta_{i,(a,b)} \neq 0\}$.
\end{definition}

There have been several attempts to describe the region of bidegrees appearing in Definition \ref{maclagansmithpreelim} in terms of the Betti numbers, with some relevant relations between the two descriptions; see \cite{botbol2012castelnuovo,bruce, chardinholanda}. \ccc{However, as shown in \cite[Example 5.1]{bruce}, there are ideals with the same bigraded Betti numbers but different regularity regions.}

In \cite{chardinholanda}, the following relation between Betti numbers and the support of local cohomology modules with respect to the ideals $\mx$ and $\my$ is established.

\begin{theorem}[{\cite[Theorem 1.2]{chardinholanda}}]\label{chardinholandator}
    Let $I$ be a bihomogeneous ideal, then
    $$\bigcup_i \beta_i(I)^{\star} \subset (n + 1, m + 1) + \big(\supp_{\mathbb{Z}^2}(H^{\bullet}_{\mx}(I))^{\star} \,\ccc{\cap}\, \supp_{\mathbb{Z}^2}(H^{\bullet}_{\my}(I))^{\star}\big).$$
\end{theorem}

In the quest of relating the properties of regularity with the degrees of the minimal generators of $\bigin(I)$ another quantity was introduced by Aramova, Crona and De Negri \cite{aramova2000bigeneric} and R\"omer \cite{roemer} by using the Betti numbers.

\begin{definition}
\label{roemerdefiniiton}
Let $I$ be a bihomogeneous ideal in $S$, then $\xtor(I)$ is the minimal degree $a \in \mathbb{Z}$ such that:
$$\beta_{i,(a'+i+1,b')}(I) = 0, $$ 
for all $ i, b' \in \mathbb{Z}_{\geq 0}$ and for all $a'\geq a$. $\ytor(I)$ is defined similarly.
\end{definition}


If $J$ is a monomial ideal satisfying the properties of Lemma \ref{lemmabiborel}, for instance $\bigin(I)$, then $\xtor(J)$ (resp. $\ytor(J)$) is the maximal degree of any minimal generator of $J$ with respect to the degrees of $x$ variables (resp. $y$).

\begin{theorem}[{\cite[Theorem 2.2]{aramova2000bigeneric}}]
\label{thm:boundAramova}
    Let $I$ be a bihomogeneous ideal. Then, there is $b \in \mathbb N$ and a generator of degree $(\xtor(\bigin(I)),b)$ in $\bigin(I)$. Moreover, no generator of $\bigin(I)$ has degree with respect to the $x$ variables bigger than $\xtor(\bigin(I))$. The same property holds for $\ytor(\bigin(I))$ and the degrees with respect to the $y$ variables.
\end{theorem}

Furthermore, R\"omer proved that, using the relative order Eq.~\eqref{eqmono}, $\xtor(I)$ behaves well with respect to the bigeneric initial ideal, i.e.
\begin{equation}\label{xtorbigin}\xtor(I) = \xtor(\bigin(I)).\end{equation}
A direct consequence of this is that the maximum degree of the \lb{minimal} generators of $\bigin(I)$ with respect to the $x$ variables is  $\xtor(I)$.

\begin{theorem}[{\cite[Proposition 4.2]{roemer}}]
\label{thm:boundRomer}
    Let $I \subset S$ be a bihomogeneous ideal. Then, there is $b \in \mathbb N$ and a generator of degree $(\xtor(I),b)$ in $\bigin(I)$. Moreover, no generator of $\bigin(I)$ has degree with respect to the $x$ variables bigger than $\xtor(I)$.
\end{theorem}

R\"omer also noted \cite[Remark 4.3]{roemer} that, as we are using the monomial order in Eq.~\eqref{eqmono}, 
$\ytor(I)$ and $\ytor(\bigin(I))$ might be different and so the previous theorem does not hold for the $y$ variables. We illustrate this in the following example.



\begin{example}
\label{example2}
We continue Example \ref{ex:2-diff-orders}. From the minimal free resolution of $I$ (see \cite[Example 7.1]{bruce}) we get $\xtor(I) = 8$ and $\ytor(I) = 3$. Using the results in \cite{aramova2000bigeneric} and \cite{roemer}, we can derive the maximal degrees of the generators $\bigin(I)$ with respect to each block of variables.
    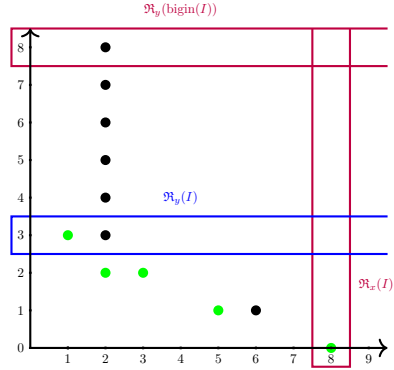
\begin{figure}
        \centering
        \input{tikzfiles/fig3.tikz}
        
        \caption{The bound for the degrees of the $x$'s is given by $\xtor(I)$. However, the bound for the degrees of the $y$'s is given by $\ytor(\bigin(I))$, which is different from  $\ytor(I)$.}
        \label{fig:figureroemerintro}
    \end{figure}
\end{example}

\section{First examples and the case of ideals defining empty varieties}
\label{sec:first_examples} 

The works of Aramova et al. \cite{aramova2000bigeneric} and R\"omer \cite{roemer} allow us to construct a \lb{bounding region} for the bidegrees of the minimal generators of $\bigin(I)$ \ccc{with respect to one of the degree components}.
One of our objectives in this work is to construct finer regions bounding these bidegrees \ccc{compared to those appearing in \cite{aramova2000bigeneric,roemer}.}

Following the work of Bayer and Stillman, a natural candidate to construct such regions is the bigraded Castelnuovo-Mumford regularity either from $I$ or $\bigin(I)$. \ccc{In this section, we illustrate that} it is not possible to construct straightforwardly such a generalization \ccc{by showing two phenomena that could never have happened in the $\mathbb{Z}$-graded case}: on the one hand, there can be unbounded regions outside of $\reg(I)$ and $\reg(\bigin(I))$ where there is no minimal generator of $\bigin(I)$ (see Example \ref{examplemat2}); on the other hand, there can be minimal generators with bidegrees which are strictly bigger than $\reg(I)$; see Example \ref{example3}. The same example also illustrates that $\reg(I)$ and $\reg(\bigin(I))$ may differ.

We start our discussion by considering the case of ideals defining empty varieties for which we can characterize the regularity in terms of the associated Hilbert function.
Let $I$ be a bihomogeneous ideal defining an empty variety of $\mathbb{P}^n \times \mathbb{P}^m$ and consider the associated Hilbert function
\begin{equation}
\label{hilbertfunction}
\begin{array}{clcl}
     \HF_{S/I} : &  \mathbb{Z}^2  &\to& \mathbb{Z}_{\geq 0}   \\
     & (a,b) &\mapsto& \dim_{\mathbf{k}}(S/I)_{(a,b)} 
\end{array}.
\end{equation}
\ccc{In the case of $I$ defining an empty variety, }the regularity $\reg(I)$ is determined by the bidegrees at which this function attains the value zero, i.e.,
\begin{equation}
    \label{remarkhilbertfunction}
 \reg(I) = \{ (a,b) \in \mathbb{Z}^2 \,:\, \HF_{S/I}(a,b) = 0 \}.
\end{equation}
\ccc{Note that $\reg(I)$ coincides with the degrees at which $H_{\mathfrak{b}}^1(I)_{(a,b)} = 0$. As, in this case, $H^i_{\mathfrak{b}}(I) = 0$ for $i \neq 1$, we can derive Eq.~\eqref{remarkhilbertfunction} from the  Grothendieck-Serre formula \cite[Proposition 2.14]{boundsmaclagansmith}.}

It is a general property that the Hilbert functions of an ideal $I$ and its bigeneric initial ideal $\bigin(I)$ coincide; this also follows similarly from the single-graded case \cite[Theorem 15.26]{eisenbud1995}. Therefore, if $I$ defines an empty subvariety of $\mathbb{P}^n \times \mathbb{P}^m$, it holds that $$\reg(I) = \reg(\bigin(I)).$$ In addition, it is possible to determine the existence of minimal generators of $\bigin(I)$ at bidegrees where the Hilbert function of $I$ vanishes.



    
%


\begin{theorem}
\label{hilbertfunctiontheorem}
Let $(a,b) \in \mathbb{Z}_{\geq 0}^2$ be a bidegree such that $\HF_{S/I}(a,b) = 0$. Then, $\bigin(I)$ has a minimal generator of bidegree $(a,b)$ \ccc{, if and only if, }$\HF_{S/I}(a',b') \neq 0$ for every $(a',b') \lneq (a,b)$.
\end{theorem}

\begin{proof}
With respect to DRL, the monomial $x_0^{a}y_0^{b}$ is the smallest monomial of bidegree $(a,b)$. As $\HF_{S/I}(a,b) = 0$, we have that $x_0^{a}y_0^{b} \in \bigin(I)_{(a,b)}$. If $x_0^{a}y_0^{b}$ is not a minimal generator of $\bigin(I)$, then there must be $(a',b') \lneq (a,b)$ such that $x_0^{a'}y_0^{b'} \in \bigin(I)$. By Lemma \ref{lemmabiborel}, this last condition is equivalent to the fact that every monomial of bidegree $(a',b')$ belongs to $\bigin(I)$, that is, equivalent to $\HF_{S/I}(a',b') = 0$ for some $(a',b') \lneq (a,b)$. 
\end{proof}

As a consequence of this theorem, we deduce that the region $\reg(I)$ does not contain bidegrees of minimal generators of $\bigin(I)$, except maybe in its minimal bi-degrees with respect to inclusion.

\begin{corollary}
\label{emptyimplication}
Assume that $I$ defines an empty variety of $\mathbb{P}^n \times \mathbb{P}^m$. If $(a,b) \in \reg(I)$ then $\bigin(I)$ has no minimal generator of bidegree $(a',b') \gneq (a,b)$.
\end{corollary}
\begin{proof}
    Follows from Eq.~\eqref{remarkhilbertfunction} and Theorem \ref{hilbertfunction}.
\end{proof}

\begin{remark} \lb{We notice that the above corollary also follows from Theorem \ref{maclagansmithgenerators} since $\reg(I)=\reg(\bigin(I))$ under the assumption that $I$ defines an empty variety in $\mathbb{P}^n \times \mathbb{P}^m$.}
\end{remark}

As mentioned before, unlike in the single graded case, the region $\reg(I)$ does not yield a sharp description of the bidegrees of a minimal set of generators of $\bigin(I)$. We illustrate it with the following example where we observe that the converse of Corollary \ref{emptyimplication} does not hold and also that $\reg(I)$ does not yield any information for infinitely many bidegrees.   

\begin{example}
\label{examplemat2}

\begin{figure}[h]
    \centering
    \input{tikzfiles/fig4.tikz}
     \caption{The green dots \textcolor{green}{$\bullet$} represent the bidegrees $(a,b)$ of the generators of the ideal $I$ in Example \ref{examplemat2}. The black dots $\bullet$ represent the bidegrees $(a,b)$ of minimal generators of $\bigin(I)$ and the white dots those bidegrees for which $\HF_{S/I}(a,b) = 0$. The region $\reg(I)$ is marked in \textcolor{red}{red}. In \textcolor{blue}{blue} (resp. \textcolor{brown}{brown}), an infinite column (resp. a row) which does not intersect $\reg(I)$.}
    \label{fig:examplemat}
\end{figure}
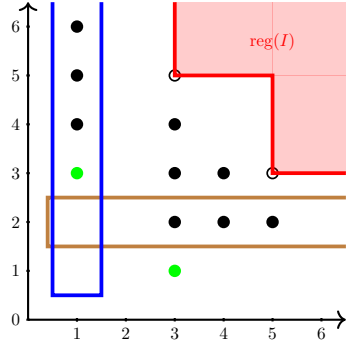
Consider the standard $\mathbb{Z}^2$-graded ring $S = \mathbb{C}[x_0,x_1,y_0,y_1]$ and the ideal $I \subset S$ generated by four bihomogeneous polynomials: 
$$ p_3(x_0,x_1)q_1(y_0,y_1), p_3'(x_0,x_1)q_1'(y_0,y_1), p_1(x_0,x_1)q_3(y_0,y_1),p_1'(x_0,x_1)q_3'(y_0,y_1),$$
where $p_i$'s and $p_i'$'s \ccc{are} general forms of degree $i$ in $x_0,x_1$ and $q_i$'s and $q_i'$'s \ccc{are} general forms of degree $i$ in $y_0,y_1$. The ideal $I$ defines an empty variety. In Figure \ref{fig:examplemat} we show the bidegrees of a minimal set of generators for $\bigin(I)$, as well the region $\reg(I)$.

At the bidegrees $(3,5)$ and $(5,3)$, Theorem~\ref{hilbertfunctiontheorem} applies so that there are minimal generators of $\bigin(I)$. The bidegree $(4,4)$, where $\HF_{S/I}(4,4) \neq 0$, so that $(4,4) \notin \reg(I)$, shows that the converse of Corollary \ref{emptyimplication} does not hold. Also, in Figure \ref{fig:examplemat} we highlighted an infinite column and an infinite row that does not intersect $\reg(I)$.
\end{example}

It turns out that if we do not restrict to ideals defining an empty variety, then even Corollary~\ref{emptyimplication} is no longer true, i.e.~there might be bidegrees $(a,b) \in \reg(I)$ such that there are minimal generators of $\bigin(I)$ in  degree $\gneq (a,b)$. 
\begin{example}
\label{example3}
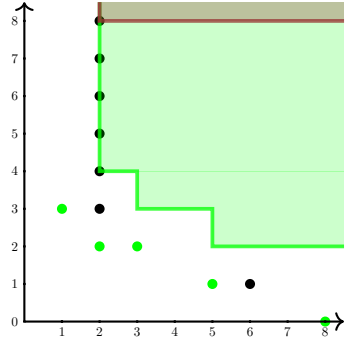
\begin{figure}[h]
    \centering

    \input{tikzfiles/fig5-2.tikz}

    \caption{The multigraded Castelnuovo-Mumford regularity $\reg(I)$, in \textcolor{green}{green} and the multigraded Castelnuovo-Mumford regularity of the bigeneric initial ideal $\reg(\bigin(I))$, in \textcolor{purple}{purple}.}
    \label{fig:examplebruce}
\end{figure}   
We continue with Example \ref{ex:2-diff-orders}. We note that $(2,4) \in \reg(I)$, nevertheless there are generators of bidegrees $\gneq (2,4)$; see Figure \ref{fig:examplebruce}.
This example also illustrates that, in general, we cannot expect that $\reg(I)$ and $\reg(\bigin(I))$ to coincide, even though we have an inclusion, as in Eq.~\eqref{inclusion}. For this reason, there is no straightforward way of using $\reg(I)$ in order to \ccc{provide a bounding region} to the bidegrees of the minimal generators of $\bigin(I)$.
\end{example}

\begin{remark}
    To compute the regions in these examples, we used the {\tt Macaulay2}  packages \texttt{VirtualResolutions}~\cite{virtualresolutionsm2} and \texttt{LinearTruncations} \cite{lineartrunc}. In both packages, the input is assumed to be saturated, which is the case of the ideal $I$ in Example \ref{ex:2-diff-orders}.
\end{remark}



\section{The partial regularity region and its main properties}
\label{sec:partial-BS}

The results and examples presented in Section \ref{sec:first_examples} illustrate the difficulty to establish a direct bihomogeneous analogue of Eq.~\eqref{eq:homo-eq-dgGB-reg} by means of the bigraded Castelnuovo-Mumford regularity region. To unravel this situation, we introduce a new region, denoted by $\xreg(I)$ and explore its properties. Compared to the Castelnuovo-Mumford regularity region, which relies on the vanishing of local cohomology modules with respect to the irrelevant ideal $\mathfrak{b}$ of the product of two projective spaces, this new region relies on the vanishing of local cohomology modules with respect to the ideal $\mx$. This construction is inspired by the work of Botbol, Chardin and Holanda \cite{botbol2012castelnuovo, chardinholanda}.

\begin{definition}
\label{def:partial-regularity}
Let $I \subset S$ be a bihomogeneous ideal. We denote by $\xreg(I)$, and call it the \textit{partial regularity region}, the region of bidegrees $(a,b) \in \mathbb{Z}^2$ such that for all $i \geq 1$ and $(a',b') \geq (a - i + 1, b)$,
$$ H_{\mx}^i(I)_{(a',b')} = 0.$$ 
\end{definition} 

One of the main results of our work is that, in generic coordinates, $\xreg(I)$ provides a bounding region to bidegrees of the elements in a minimal Gr\"obner basis which is (partially) tight. This result will be proved in Section \ref{sec::5}. In order to do so, in this section we need to establish two key properties of $\xreg(I)$. First, in \ccc{Theorem \ref{theoremcriterion}}, we present a criterion to \ccc{characterize} $\xreg(I)$ similar to the one proposed by Bayer and Stillman in the classical setting \cite[Theorem 1.10]{bayer_criterion_1987}. 
Second, in Theorem \ref{main-theo-sec-4}, we show that, in generic coordinates, the partial regularity region of an ideal and its bigeneric initial ideal agree, that is, $\xreg(I) = \xreg(\bigin(I))$.
The results on this section generalize the ones in \cite[\S 1]{bayer_criterion_1987} and most of the proofs follow similar strategies.

\subsection{A criterion for the partial regularity region}








\begin{notation} Let $I \subset S$ be a bihomogeneous ideal.
\begin{itemize}
    \item[-] We denote by $I^{\xsat}$ the saturation of $I$ with respect to $\mx$, i.e. $(I:\mathfrak{m}_{x}^{\infty})$.
    \item[-] Given any polynomial $f \in S$, $(I,f)$ will denote the sum of the ideals $I$ and $(f)$.
    \item[-] We will denote by $\mathbb{Z}^2_{\succsim 0}$ the product $\mathbb{Z}_{>0} \times \mathbb{Z}_{\geq 0}$. 
\end{itemize}
\end{notation}


\begin{lemma}
\label{lemmasaturated}
    A generic linear $x$-form $h$ is not a zero divisor in $S/I^{\xsat}$. Namely, $(I^{\xsat}:h) = I^{\xsat}$.
\end{lemma}

\begin{proof}
    The proof follows using the same argument as \cite[Lemma 3.3]{vantuyl}.
\end{proof}



The following lemma shows that local cohomology modules with respect to $\mx$ vanish when the degree with respect to the $x$ variables is big enough.


\begin{lemma}
\label{vanishingoflocalcohomology}
Let $I \subset S$ be a bihomogeneous ideal. Then there is $a_0 \in \mathbb{Z}$ such that for all $b \in \mathbb{Z}$, it holds 
    $$H^i_{\mx}(I)_{(a,b)} = 0 \quad 
   \forall a \geq a_0. 
    $$
\end{lemma}

\begin{proof}

It is classical property that the local cohomology modules $H^i_{\mx}(I)$ can be defined and computed using the \v{C}ech complex $\mathcal{C}_{\mx}^{\bullet}(I)$. We refer the reader to \cite{broadmann} for more details on the construction of this complex and its main properties. In this proof, we will use this complex, together with a minimal free resolution $F_{\bullet}$ of $I$, to construct a double complex that we denote by $\mathcal{C}_{\mx}^{\bullet}(F_{\bullet})$. We note that this double complex has often been used in the bibliography; see, for example, \cite[\S 2]{Boua_Hassanzadeh_2019}.

There are two natural spectral sequences associated with the double complex $\mathcal{C}_{\mx}^{\bullet}(F_{\bullet})$, depending on whether we consider the filtrations with respect to the horizontal or the vertical maps. Both sequences converge to the same limit \cite[Section 1.2]{weibel1994introduction}. Considering the filtration given by the horizontal maps, since $F_{\bullet}$ is a minimal free resolution of $I$ we deduce that the spectral sequence converges to $H^{\bullet}_{\mx}(I)$ in its second page. 

Similarly, the second spectral sequence has the terms $H^i_{\mx}(F_q)$ in its first page. These terms are direct sums of $H^i_{\mx}(S)$, up to the shifts appearing in the minimal free resolution $F_{\bullet}$. \ccc{Using Eq.~\eqref{localcohomologyx}, we deduce that there exists $a_0 \in \mathbb{Z}$ such that for all $b \in \mathbb{Z}$,
$$H^i_{\mx}(F_q)_{(a',b')} = 0 \text{ for all }i,q \text{ and } (a',b') \geq (a_0,b) .$$}
\ccc{Since both spectral sequences converge to the same limit, we deduce the isomorphism:
$$H_{\mx}^q(I)_{(a',b')} \cong \ker \oplus_{p}(H^{p+i}_{\mx}(F_p) \xrightarrow[]{} H^{p+i}_{\mx}(F_{p-1}))_{(a',b)} = 0$$ for all $i,q$ and $(a',b') \geq (a_0,b)$.}
\end{proof}



\begin{lemma}
\label{lemmageneric}
    Let $I\subset S$ be a bihomogeneous ideal. Let $h$ be generic linear $x$-form and $(a,b) \in \mathbb{Z}_{\geq 0}^2$. Then, the following are equivalent:
    \begin{itemize}
        \item[i)]  $(I:h)_{(a',b')} = I_{(a',b')}$ for all $(a',b') \geq (a,b)$.
        \item[ii)] $H_{\mx}^1(I)_{(a',b')} = 0$ for all $(a',b') \geq (a,b)$.
    \end{itemize}
\end{lemma}

\begin{proof}
Since $n\geq 1$, we have that $H_{\mx}^1(I) \cong H_{\mx}^0(S/I) = I^{\xsat}/I$; see for example \cite[Remark 1.3]{bayer_criterion_1987}. Hence, $H_{\mx}^1(I)_{(a',b')} = 0$, if and only if, $I^{\xsat}_{(a',b')} = I_{(a',b')}$.
By  Lemma \ref{lemmasaturated}, as $h$ is a generic $x$-form, condition $ii)$ implies condition $i)$ as in this case we have that, for every $(a',b') \geq (a,b)$
$$ (I:h)_{(a',b')} = (I^{\xsat}:h)_{(a',b')} = I^{\xsat}_{(a',b')} = I_{(a',b')}.$$

To prove the opposite implication, we observe that by Lemma \ref{vanishingoflocalcohomology}, for a given bidegree $(a,b) \in \mathbb{Z}^2$, there is a $\lambda_0 \in \mathbb{Z}_{\geq 0}$ such that, for every $\lambda \geq \lambda_0$, we have that
$$I_{(a' + \lambda,b')} = I^{\xsat}_{(a' + \lambda,b')} \text{ for every } (a',b') \geq (a,b).$$
Either the previous condition holds for every $\lambda_0$ and so $H^1_{\mx}(I)_{(a',b')} = 0$ for all $(a',b') \geq (a,b)$, or there is a minimal $\lambda_0$ satisfying the previous condition.
In the latter case,
by minimality of $\lambda_0$, we have that
$I_{(a' + \lambda_0 - 1,b')} \neq I^{\xsat}_{(a' + \lambda_0 -1,b')}$ for some $(a',b') \geq (a,b)$. 
Therefore, there must be bihomogeneous $f \in I^{\xsat}$ of bidegree $(a' +\lambda_0 - 1,b')$ such that $f \notin I$.
However, as $I^{\xsat}_{(a' + \lambda_0,b')} = I_{(a' + \lambda_0,b')}$, 
for every $x$-form $h \in S_{(1,0)}$, we have that $h \, f \in I_{(a' + \lambda_0,b')}$ and so $f \in (I:h)_{(a' + \lambda_0 - 1,b')}$.
If $\lambda_0 \geq 1$, condition $i)$ implies that $(I:h)_{(a' + \lambda_0 - 1,b')} = I_{(a' + \lambda_0 - 1,b')}$, so we get a contradiction as 
$f \not\in I$. 
Hence $\lambda_0 \leq 0$ and so $H_{\mx}^1(I)_{(a',b')} = 0$ for every $(a',b') \geq (a,b)$. 
\end{proof}

The following lemma shows that the partial regularity region can be computed recursively using colon ideals with respect to generic linear $x$-forms, and \ccc{is a bigraded analogue of \cite[Theorem 1.8]{bayer_criterion_1987}. In the proof we use the results of Chardin and Holanda \cite{chardinholanda}.} 

\begin{lemma}
\label{lemmainduction}
   Let $h$ be a generic linear $x$-form and $(a,b) \in \mathbb{Z}_{\geq 0}^2$, then the following are equivalent.
    \begin{itemize}
        \item[i)] $(a,b) \in \xreg(I)$.
        \item[ii)] $(I:h)_{(a',b')} = I_{(a',b')}$ for $(a',b') \geq (a,b)$ and $(a,b) \in \xreg(I,h)$.
    \end{itemize}
\end{lemma}

\begin{proof}
    First, we observe that, 
    if $(I:h)_{\geq (a,b)} = I_{\geq (a,b)}$, then for every $i\geq 1$,
    $ H_{\mx}^i((I:h)_{\geq (a,b)}) = H_{\mx}^i(I_{\geq (a,b)}).$
    Under this assumption, Lemma \ref{lemmachardinholandasupports} implies that
   \begin{equation}
    \label{equation1}
        H_{\mx}^1(I:h)_{(a',b')} = H_{\mx}^1(I)_{(a',b')} \text{ for every }  (a',b') \geq (a,b).
    \end{equation}
    and
\begin{equation}
    \label{equation2}
        H_{\mx}^i(I:h) = H_{\mx}^i(I) \text{ for every }i \geq 2.
    \end{equation}
    To prove that $i)$ implies $ii)$, let $(a,b) \in \xreg(I)$. By Lemma \ref{lemmageneric}, we get that $(I:h)_{\geq (a,b)} = I_{\geq (a,b)}$, so that Eq.~\eqref{equation1} and Eq.~\eqref{equation2} hold.
    Consider the short exact sequence which is induced by the multiplication by $h$,
    \begin{equation}
    \label{equation3}
        0 \xrightarrow[]{} (I:h)(-(1,0)) \xrightarrow[]{} I \oplus (h) \xrightarrow[]{} (I,h) \xrightarrow[]{} 0.
    \end{equation}
    For every $i\geq 1$, taking the graded components of the corresponding long exact sequence of local cohomology at degrees $(a',b') \geq (a - (i-1),b)$ yields  
\begin{equation}\label{equation4} \cdots \xrightarrow[]{} H^{i}_{\mx}(I)_{(a',b')} \xrightarrow[]{} H^i_{\mx}(I,h)_{(a',b')} \xrightarrow[]{} H^{i+1}_{\mx}(I:h)_{(a'-1,b')} \xrightarrow[]{} \cdots\end{equation}
We notice that the last term in Eq.~\eqref{equation4} can be replaced by $H^{i+1}_{\mx}(I)_{(a'-1,b')}$ as $i\geq 1$, using Eq.~\eqref{equation2}. Moreover, as we are assuming that $(a,b) \in \xreg(I)$, the two graded components of the local cohomology of $I$ in Eq.~\eqref{equation4} vanish, and so $H^i_{\mx}(I,h)_{(a',b')} = 0$ for all $(a',b') \geq (a - (i-1),b)$.

In order to prove that condition $ii)$ implies condition $i)$, we consider $(a,b) \in \xreg(I,h)$ such that $(I:h)_{(a',b')} = I_{(a',b')}$ for every $(a',b') \geq (a,b)$. By Lemma \ref{lemmasaturated}, we have that $H^1_{\mx}(I)_{(a',b')} = 0$ for $(a',b') \geq (a,b)$.
As we did above, we consider the long exact sequence associated to Eq.~\eqref{equation3} at the graded pieces given by $(a',b') \geq (a - (i-2),b)$,
\begin{multline*}
   \cdots \xrightarrow[]{} H^{i-1}_{\mx}(I,h)_{(a',b')} \xrightarrow[]{} H^i_{\mx}(I:h)_{(a'-1,b')} \\
   \\ \xrightarrow[]{\delta_i} H^i_{\mx}(I)_{(a',b')} \xrightarrow[]{} H^i_{\mx}(I,h)_{(a',b')} \xrightarrow[]{} \cdots.  
\end{multline*}
As $(a,b) \in \xreg(I,h)$, the graded pieces of local cohomology modules associated to $(I,h)$ vanish, and so the map $\delta_i$ is an isomorphism. 
By Lemma \ref{vanishingoflocalcohomology}, there exists $\lambda$ sufficiently big such that
\begin{equation}
\label{equation5}
    H^i_{\mx}(I)_{(a' + \lambda,b')} = 0 \quad (a',b') \geq (a - (i-1),b).
\end{equation}
As we are assuming that $(I:h)_{\geq (a,b)} = I_{\geq (a,b)}$,  Eq.~\eqref{equation2} holds and, together with Eq.~\eqref{equation5} and  the isomorphism $\delta_i$, we have that for every $i > 1$ and every $(a',b') \geq (a - (i-1),b)$,
\begin{equation}
\label{equation6}
    H^i_{\mx}(I)_{(a' + \lambda -1,b')} \underrel{\text{}}{\cong}  H^i_{\mx}(I:h)_{(a' + \lambda -1,b')} \underrel{\delta_i}{\cong} H^i_{\mx}(I)_{(a' + \lambda,b')} \underrel{}{=} 0.
\end{equation}
where the first isomorphism follows from Eq.~\eqref{equation2}, the second isomorphism is $\delta_i$, and the last equality to zero follows from Eq.~\eqref{equation5}.
If we apply the above repeatedly, starting from big enough $\lambda$, we conclude that for every $i \geq 1$,  $H_{\mx}^i(I)$ vanishes at every degree bigger or equal to $(a - (i-1), b)$, and therefore $(a,b) \in \xreg(I)$.
\end{proof}


The following theorem, which aims to characterize $\xreg(I)$, can be seen as a partial extension of the criterion of Bayer and Stillman to \ccc{characterize} the Castelnuovo-Mumford regularity in the single graded case \cite[Theorem 1.10]{bayer_criterion_1987} to the setting of bigraded ideals.

\begin{theorem}
\label{theoremcriterion}
Consider a bihomogeneous ideal $I \subset S$. Then, for $(a,b) \in \mathbb{Z}^2_{\succsim 0}$, the following are equivalent:
\begin{itemize}
    \item[i)] $(a,b) \in \xreg(I)$.
    \item[ii)] There exists a non-negative integer $k_0 \leq n$ such that for all $k = 0,\dots,k_0$ and $(a',b') \geq (a,b)$, we have:
    \begin{equation}
    \label{coloncondition}(J_{k-1}:h_k)_{(a',b')} = (J_{k-1})_{(a',b')}\end{equation}
    where $J_{k-1} = (I,h_0,\dots,h_{k-1})$ (with the convention $J_{-1} = I$), $h_k$ are generic linear $x$-forms and $J_{k_0} \supset \mx$.
\end{itemize}
\end{theorem}

\begin{proof} 

We will proceed by induction in the minimal number $k_0$ such that 
$J_{k_0} \supset \mx$. Consider $I$ such that $k_0 = -1$. 
As $I \supset \mx$, we have $I^{\xsat} = S$.
Moreover, as in the classical setting (see \cite[Corollary 2.1.7]{broadmann})
the higher local cohomology modules of $I$ are the same as those of the saturation, and so
$$H^i_{\mx}(I) = H^i_{\mx}(I^{\xsat}) = H^i_{\mx}(S)  \text{ for every } i \geq 2.$$
By Remark~\ref{rem:lcohomo-when-I-eq-S}, the previous cohomology $H^i_{\mx}(S)_{(a,b)}$ vanish, unless $i = n + 1$, $a \leq -n-1$ and $b \geq 0$.
Therefore,
$$ H^{n+1}_{\mx}(I)_{(a',b')} = 0 \quad \text{ for } (a',b') \geq (a-n,b).$$
Moreover, as $a > 0$ and $I \supset \mx$, we have $I^{\xsat}_{(a,b)} = I_{(a,b)}=S_{(a,b)}$. This last condition is equivalent to the fact that $H^{1}_{\mx}(I)_{(a,b)} = 0$ for $a > 0$, and so  $\xreg(I) \supset \mathbb{Z}^2_{\succsim 0}$.
This shows that, for $k_0 = -1$ condition $i)$ is also always satisfied.

For the inductive step, assume that the theorem holds for every ideal such that its associated $k_0$ is at most $t$.
Consider an ideal $I$ such that its associated $k_0$ is $t+1$. Then, for a generic $x$-form $h$, the  $k_0$ associated to $(I,h)$ is $t$. 
Hence, we can apply our inductive hypothesis to $(I,h)$.
The proof follows straightforwardly from Lemma~\ref{lemmainduction}.
\end{proof}

\ccc{\begin{remark}
The above generalization of the work of Bayer and Stillman is partial in the following sense: in \cite[Theorem 1.10]{bayer_criterion_1987}, they show that, in the standard $\mathbb{Z}$-graded case, if Eq.~\eqref{coloncondition} is satisfied at degree $a$, then it is also satisfied for all degrees $a' \geq a$. In order to prove this, they need to use that if $a \in \reg(I,h)$ for a generic linear form $h$ and $I$ is generated in degrees lower than $a$, then $(I:h)$ is generated in degrees lower than $a$ (see \cite[Lemma 1.9]{bayer_criterion_1987}). Further work is required in order to find an analogue of \cite[Lemma 1.9]{bayer_criterion_1987} for $\xreg(I,h)$ that allows us show that if Eq.~\eqref{coloncondition} is satisfied at $(a,b) \in \mathbb{Z}^2$, then it is satisfied for all $(a',b') \geq (a,b)$. In a recent paper, we treat this problem in the special case of systems where the first projection of the biprojective variety is zero-dimensional; see \cite[Theorem. 2.3]{ourPaperAtISSAC}.
\end{remark}
}

\subsection{The partial regularity region and the bigeneric initial ideal} Our next goal is to prove that the region $\xreg(I)$ and the region $\xreg(\bigin(I))$ coincide. This will only happen if we consider the $\bigin(I)$ with respect to the $\DRL$ monomial order in Eq.~\eqref{eqmono}. In the next lemma, we analyze the behavior of ideals under change of coordinates.


\begin{lemma}
\label{lemmagenericform}
Let $I \subset S$ be a bihomogeneous ideal and $u \in \GL(n+1) \times \GL(m+1)$.
Then, the following hold:
\begin{itemize}
    \item[i)] $u \circ (I^{\xsat}) = (u \circ I)^{\xsat}$.
    \item[ii)] Let $h_0,\dots,h_{n}$ be linear forms satisfying that $x_{k} = u \circ h_{k}$ for $k = 0,\dots,n$ and $(a,b) \in \mathbb{Z}^2$. Then,\begin{multline*}
        (I,h_0,\dots,h_{k-1}:h_{k})_{(a,b)} = (I,h_{0},\dots,h_{k-1})_{(a,b)} \iff  \\ [(u \circ I),x_{0},\dots,x_{k-1}:x_k]_{(a,b)} = [(u \circ I),x_{0},\dots,x_{k-1}]_{(a,b)}.
    \end{multline*}
    for all $k = 0,\dots,n$.
\end{itemize}
\end{lemma}

\begin{proof}
    In order to prove part $i)$, we note that $\mx$ is invariant under the action of $\GL(n+1) \times \GL(m+1)$. Therefore, if $f \in u \circ I^{\xsat}$, then there is $g \in I^{\xsat}$ such that $f = u \circ g$ and there exists $t$ with $g\mathfrak{m}_x^t \subset I$. This is equivalent to the fact that $f\mathfrak{m}_x^{t} \subset u \circ I$. Similarly, the proof of part $ii)$ follows from the fact that the colon ideal commutes with the change of coordinates.
\end{proof}

The following lemma shows that we can verify the equality between $I$ and its colon ideal with respect to a variable by looking at the initial ideal. This is a classical property of the $\DRL$ monomial order defined in Eq.~\eqref{eqmono}; see Remark \ref{remarkofdrl} for more details.

\begin{lemma}[{\cite[Lemma 2.2]{bayer_criterion_1987}}]
\label{biglemma} Let $I \subset S$ be a bihomogeneous ideal and let $(a,b) \in \mathbb{Z}_{\succsim 0}^2$. For $k = 0,\dots,n$, we have the following:
\begin{itemize}
    \item[i)] $\ini(I,x_0,\dots,x_k) = (\ini(I),x_0,\dots,x_k)$.
    \item[ii)] Suppose that $x_0,\dots,x_{k-1} \in I$ and that we are using the $\DRL$ monomial order in Eq.~\eqref{eqmono}, then: $$   (I:x_k)_{(a,b)} = I_{(a,b)} \iff  (\ini(I):x_k)_{(a,b)} = \ini(I)_{(a,b)}. $$ 
\end{itemize}
\end{lemma}
The following lemma generalizes Lemma \ref{lemmasaturated} to the bigeneric initial ideal.

\begin{lemma}
\label{nonzerodivisor_in_bigin}
    Let $I \subset S$ be a bihomogeneous ideal. For every $k = 0,\dots,n$, let $J_{k-1} := (\bigin(I),x_0,\dots,x_{k-1})$ (with the convention $J_{-1}=I$). If $(J_{k-1})^{\xsat} \neq S$, then $x_k$ is a non-zero divisor in $S/(J_{k-1})^{\xsat}$.
\end{lemma}

\begin{proof}
     We first prove the case $k = 0$. Following the same argument as in \cite[Lemma 2.1]{concauniversal}, we note that the associated primes of a bi-Borel fixed ideal are of the form:
    $$ \mathcal{P}_{t} = (x_{t_x},\dots,x_{n},y_{t_y},\dots,y_m)$$
    for some $t_x,t_y \in \mathbb{Z}^2$ such that $0 \leq t_x \leq n$ and $0 \leq t_y \leq m$. If $\bigin(I)^{\xsat} \neq S$, the associated primes of $\bigin(I)^{\xsat}$ must satisfy $t_x > 0$. Therefore, $x_{0}$ cannot be contained in the union of the associated primes of $\bigin(I)^{\xsat}$.

    In the case where $k > 0$, we note that $\bigin(I) \cap \mathbf{k}[x_k,\dots,x_n,y_0,\dots,y_m]$ is also bi-Borel fixed. Therefore, the associated primes of $J_{k-1}$ which contain $x_k$ are not associated primes of $(J_{k-1})^{\xsat}$. The proof follows by the same argument as above.
\end{proof}


Using the above lemmas, we obtain the following theorem. 
\begin{theorem}
\label{main-theo-sec-4}
Let $I \subset S$ be a bihomogeneous ideal and $h_0,\dots,h_k$ are generic linear $x$-forms. Then, for every $(a,b) \in \mathbb{Z}^2_{\succsim 0}$ and $k = 0,\dots,n$
\begin{multline}
\label{eqbigin}
    (I,h_0,\dots,h_{k-1}:h_{k})_{(a,b)} = (I,h_{0},\dots,h_{k-1})_{(a,b)} \iff \\ \quad [(\bigin(I),x_{0},\dots,x_{k-1}):x_k]_{(a,b)} = [(\bigin(I),x_{0},\dots,x_{k-1})]_{(a,b)}.
\end{multline} 
In particular, $\xreg(I) \cap \mathbb{Z}_{\succsim 0}^2 = \xreg(\bigin(I)) \cap \mathbb{Z}_{\succsim 0}^2$.
\end{theorem}



\begin{proof}
   The first part of the proof follows straightforwardly from Lemma \ref{lemmageneric} and Lemma \ref{biglemma}. For the second part, we note that Lemma \ref{nonzerodivisor_in_bigin} implies that the proof of Theorem \ref{theoremcriterion} can be reproduced for $\bigin(I)$ using the variables $x_0,\dots,x_n$ instead of generic linear $x$-forms $h_0,\dots,h_n$. Namely, $(a,b) \in \xreg(\bigin(I)) \cap \mathbb{Z}^2_{\succsim 0}$, if and only if, there is $k_0 \in \mathbb{Z}_{\geq 0}$ such that for all $k = 0,\dots,k_0$ and $(a',b') \geq (a,b)$, we have:
    $$(J_{k-1}:x_k)_{(a',b')} = (J_{k-1})_{(a',b')}$$
    where $J_{k-1} = (\bigin(I),x_0,\dots,x_{k-1})$
   and $J_{k_0} \supset \mx$. Therefore, the proof follows straightforwardly from the first part, i.e. from Eq.~\eqref{eqbigin}.
\end{proof}




\begin{remark}
\label{remarkofdrl}
  In Lemma \ref{biglemma}, the proof of the fact that for any $k = 0,\dots,n$ and for any $I$ such that if $x_0,\dots,x_{k-1} \in I$, we have
    \begin{equation}
        \label{refini}(\ini(I):x_k)_{(a,b)} = \ini(I)_{(a,b)} \implies (I:x_k)_{(a,b)} = I_{(a,b)}
    \end{equation}
    does not require that the monomial order $<$ is degree reverse lexicographical. Therefore, Eq.~\eqref{refini} also holds for any other monomial order. On the other hand, in \cite{converse}, Loh proved that for any monomial order different than $\DRL$, it is possible to find an ideal $I$ such that the converse implication to Eq.~\eqref{refini} does not hold, regardless of the bigraded context. This motivates our choice of using the $\DRL$ monomial order in our study of the generalization of the Bayer-Stillman criterion to the bigraded setting.
\end{remark}

    As we noticed in Example \ref{ex:2-diff-orders}, the relative order of the variables of different blocks will change the bidegrees of the generators of $\bigin(I)$. Theorem \ref{main-theo-sec-4} relies on the specific choice of this order in Eq.~\eqref{eqmono}. While the criterion in Theorem \ref{theoremcriterion} would also hold symmetrically for $\yreg(I)$, this region does not remain invariant under $\bigin(I)$ unless we change the relative order of the blocks of variables. \ccc{In the case of intermixing the two blocks of variables, the proof of Lemma \ref{biglemma} does not hold.}

\begin{example}

  We continue with Example \ref{ex:2-diff-orders} and draw the regions $\yreg(I)$ and $\yreg(\bigin(I))$. We note that, using the monomial order Eq.~\eqref{eqmono},  they are  different.

    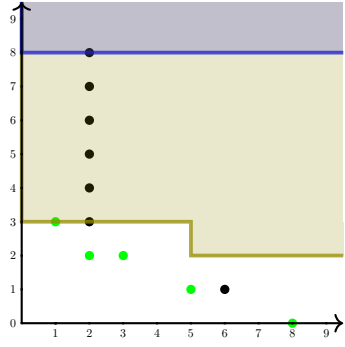
\begin{figure}
        \centering
        \input{tikzfiles/fig6.tikz}
        \caption{In \textcolor{olive}{olive}, the region $\yreg(I)$. In \textcolor{blue}{blue}, the region $\yreg(\bigin(I))$.}
        \label{fig:yreg}
    \end{figure}

    \end{example}

\section{The partial regularity region and the minimal generators of $\bigin(I)$}
\label{sec::5}

In the previous section, we provided the definition and main properties of the partial regularity region $\xreg(I)$, including a criterion which generalizes the classical result of Bayer and Stillman to the setting of regions of bidegrees that we are studying. In this section, we exploit
$\xreg(I)$ 
to prove the absence of minimal generators of $\bigin(I)$ at some bidegrees (see Theorem~\ref{theoremxreg}) and to certify that there are generators near the border of the region $\xreg(I)$ (Theorem \ref{theoremexistenceelementsxreg}). 
Moreover, we also provide relations between $\reg(I)$, $\xreg(I)$ and the Betti numbers of $I$ by relying on results by Chardin and Holanda  \cite{chardinholanda}.

The following lemma, which is the bigraded analogue of \cite[Lemma 2.2 iii)]{bayer_criterion_1987}, provides sufficient conditions for the absence of minimal generators of bidegree $(a,b)$.

\begin{lemma}
\label{lem::conditions}
Consider a bihomogeneous ideal $I \subset S$ and $k \in \{0,\dots,n\}$ such that $x_0,\dots,x_{k-1} \in I$. Let $(a,b) \in \mathbb{Z}^2_{\succsim 0}$ with $a > 1$. Assume that there is no minimal generator of $\ini(I,x_k)$ of bidegree $(a,b) \in \mathbb{Z}^2_{\succsim 0}$ and that  \begin{equation}
\label{conditions1}
(\ini(I):x_k)_{(a-1,b)} = (\ini(I) + \my(\ini(I):x_k))_{(a-1,b)}.\end{equation} Then, there is no minimal generator of $\ini(I)$ of bidegree $(a,b)$.
\end{lemma}



\begin{proof}


Consider an element $f \in I_{(a,b)}$. If  $x_0,\dots,x_{k-2}$, or $x_{k-1}$ divides $\ini(f)$, then $f$ cannot be a minimal generator of $\ini(I)$. Thus, up to substracting multiples of $x_0,\dots,x_{k-1}$, we may assume that $f \in \mathbf{k}[x_k,\dots,x_{n},y_0,\dots,y_m]$. If $x_k$ divides $\ini(f)$, then $\ini(f) = x_k\ \ini(\overline{f})$ for some $$\ini(\overline{f}) \in  (\ini(I):x_k)_{(a-1,b)} = (\ini(I) + \my(\ini(I):x_k))_{(a-1,b)}.$$
Hence, there is a non-constant $x^\alpha y^{\beta} \in (x_k,\my)$ and $l \in I$ of bidegree strictly smaller than $(a,b)$ such that $\ini(f) = x^\alpha y^{\beta} \, \ini(l)$. Therefore, $\ini(f)$ cannot be a minimal generator.

    Suppose now that $x_k$ does not divide $\ini(f)$. As there is no generator of $\ini(I,x_k)$ of bidegree $(a,b)$ and  $\ini(f) \in \ini(I,x_k)$, then we can write $\ini(f) = x^{\alpha'}y^{\beta'}\ini(g)$ with $x^{\alpha'}y^{\beta'} \neq 1$ and $\ini(g) \in \ini(I,x_k)$.
    Write $g$ as $g = g_1 + x_kg_2$ for $g_1 \in I$. Since $\ini(g) > \ini(x_kg_2)$, we have that $\ini(f) = x^{\alpha'}y^{\beta'}\ini(g_1)$ with $g_1 \in I$ an element of strictly lower bidegree than $(a,b)$. Hence, $\ini(f)$ is not a generator of $\ini(I)$
\end{proof}

Applying repeatedly Lemma~\ref{lem::conditions},  we get sufficient conditions for the absence of minimal generators of $\ini(I)$ of bidegree $(a,b) \in \mathbb{Z}_{\succsim 0}^2$.

 \begin{corollary*}
\label{bigtheorem}
    Let $I \subset S$ be a bihomogeneous ideal. Let $(a,b) \in \mathbb{Z}_{\succsim 0}^2$ with $a > 1$, and assume that:
\begin{equation}\label{equation-5-2}(\ini(J_{k-1}):x_k)_{(a-1,b)} = \big(\ini(J_{k-1}) + \my(\ini(J_{k-1}):x_k)\big)_{(a-1,b)}, \end{equation}
for all $k = 0,\dots,n$ and $J_k = (I,x_0,\dots,x_k)$ (with the convention $J_{-1}=I$). Then, there is no \lb{minimal} generator of $\ini(I)$ of bidegree $(a,b)$.
\end{corollary*}
\begin{proof}
    Note that there are no minimal generators of $\ini(I,x_0,\dots,x_n)$ of any bidegree $(a,b) \in \mathbb{Z}^2_{\succsim 0}$ as each of them must be divided by some $x_i$. By Lemma \ref{biglemma} and the hypothesis, this implies that there is no generator of $\ini(I,x_0,\dots,x_{n - 1})$ of bidegree $(a,b)$. Applying Lemma \ref{biglemma} $iii)$ recursively, we get that there is no generator of $\ini(I)$ of bidegree $(a,b) \in \mathbb{Z}^2_{\succsim 0}$.
\end{proof}

%

Applying Theorem \ref{theoremcriterion}, we derive the following result.

\begin{theorem}
\label{theoremxreg}
Let $I \subset S$ be a bihomogeneous ideal. Let $(a,b) \in \xreg(I) \cap \mathbb{Z}^2_{\succsim 0}$. If $(a',b') \geq (a+1, b)$, then there is no minimal generator of $\bigin(I)$ of bidegree $(a',b')$.
\end{theorem}

\begin{proof}Note that for every $(a',b') \in \mathbb{Z}_{\succsim 0}^2$, the equality
$$(J_{k-1}:x_k)_{(a',b')} = (J_{k-1})_{(a',b')}$$
implies that  $$(J_{k-1}:x_k)_{(a',b')} = [J_{k-1} + \my(J_{k-1}:x_k)]_{(a',b')}$$ for every $J_{k-1} = (\bigin(I),x_0,\dots,x_{k-1})$ with $k = 0,\dots,n$. Therefore, applying Corollary \ref{bigtheorem} to $\bigin(I)$ and Theorem \ref{theoremcriterion}, we deduce that if $(a,b) \in \xreg(I) \cap \mathbb{Z}^2_{\succsim 0}$, then $(J_{k-1}:x_k)_{(a',b')} = (J_{k-1})_{(a',b')}$ for all $(a',b') \geq (a,b)$ and $k = 0,\dots,n$ 
\end{proof}

In addition, we can use Lemma~\ref{lemmabiborel} to attest the presence of generators of some bidegrees, using the same criterion as in Theorem~\ref{theoremcriterion}.

\begin{theorem}
\label{theoremexistenceelementsxreg}
    Let $(a,b) \in \mathbb{Z}_{\succsim 0}^2$ with $a > 1$ such that $(a,b) \in \xreg(I)$, but $(a-1,b) \notin \xreg(I)$. Then, there exists $b' \leq b$ such that there is a minimal generator of $\bigin(I)$ of bidegree $(a,b')$.
\end{theorem}
\begin{proof}
    If $(a-1,b) \notin \xreg(I)$, then by Theorem \ref{theoremcriterion} and Eq.~\eqref{eqbigin}, we can derive that there is $0 \leq k \leq n$ such that we have $ (J_{k-1}:x_{k})_{(a-1,b)} \neq (J_{k-1})_{(a-1,b)}$ for $J_{k-1} = (\bigin(I),x_0,\dots,x_{k-1})$. This result implies that there is a monomial $x^{\alpha}y^{\beta} \in \ccc{S_{(a-1,b)}}$ such that \begin{multline}
    \label{eq3}
        x_{k}x^{\alpha}y^{\beta} \in (\bigin(I),x_{0},\dots,x_{k-1})_{(a,b)} \text{ but } \\ x^{\alpha}y^{\beta} \notin (\bigin(I),x_{0},\dots,x_{k-1})_{(a-1,b)}.
    \end{multline} Therefore, none of the variables $x_{0},\dots,x_{k-1}$ divides the monomial $x^{\alpha}y^{\beta}$. If $x_kx^{\alpha}y^{\beta}$ is a minimal generator of $\bigin(I)$, we are done. Otherwise, write $x_kx^{\alpha}y^{\beta} = \overline{z}z^{\gamma}$
    where $z^{\gamma}$ is a minimal generator of $\bigin(I)$. We need to show that the bidegree of $z^{\gamma}$ is $(a,b')$ for some $b' \leq b$. If this is not true, then there is some $k' \geq k$ such that $x_{k'}$ divides $\overline{z}$. At this point, we have two cases:
    \begin{itemize}
        \item[-] If $k' = k$ then $x_{k}x^{\alpha}y^{\beta} = x_{k} \frac{\overline{z}}{x_{k}}z^{\gamma}$, which implies that $x^{\alpha}y^{\beta} = \frac{\overline{z}}{x_{k}}z^{\gamma} \in \bigin(I)$, in contradiction with Eq.~\eqref{eq3}.
        \item[-] If $k' > k$ then $x_{k}$ divides $z^{\gamma}$ and $x_{k'}$ divides $\overline{z}$. In this case, we write $z^{\gamma} = x_{k}z^{\gamma'}$ and $\overline{z} = x_{k'}\overline{z}'$. Using the property of $\bigin(I)$ in Lemma \ref{lemmabiborel}, we get $x_{k'}z^{\gamma'} \in \bigin(I)$
        and so $x^{\alpha}y^{\beta} = x_{k'}\overline{z}'z^{\gamma'} \in \bigin(I)$,  which is in contradiction with Eq.~\eqref{eq3}.
    \end{itemize}
    Therefore, $z^{\gamma}$ has bidegree $(a,b')$ for some $b' \leq b$.
\end{proof}

\begin{example}
\label{examplebrucexreg}
Consider the ideal $I$ in Example \ref{example3}. In Figure \ref{fig:examplebrucexreg}, one shows the region $\xreg(I) + (1,0)$ where there cannot be any generators of $\bigin(I)$ (using Theorem \ref{theoremxreg}). Moreover, we mark the columns and squares in which Theorem \ref{theoremexistenceelementsxreg} guarantees that there must be minimal generators of $\bigin(I)$ of such bidegrees. Due to the vanishing of $H^i_{\mx}(I)_{(a,b)}$ for $a \gg 0$, we note that the region $\xreg(I)$ always provides a tight bound for the degrees of the generators of $\bigin(I)$ with respect to the $x$'s. \ccc{In comparison to Example \ref{example2}, where the bound provided is $8$, we see that $\xreg(I)$ contains every bidegree in $\mathbb{Z}^2_{\succsim 0}$ whose first component greater or equal than $8$. Thus,} the bound provided in \cite[Proposition 4.2]{roemer} is also recovered.
    \begin{figure}
        \centering
        \input{tikzfiles/fig7.tikz}
         \caption{In \textcolor{brown}{brown}, the region $\xreg(I) + (1,0)$. In \textcolor{blue}{blue}, columns and squares where there are generators of $\bigin(I)$.}
        \label{fig:examplebrucexreg}
    \end{figure}
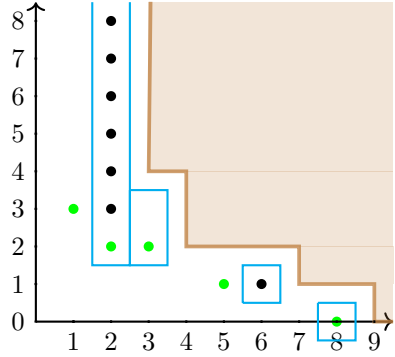
\end{example}

In what follows, we study the relation between $\xreg(I)$, 
the multigraded Castel\-nuovo-Mumford regularity of $I$ and the bidegrees of the generators of $\bigin(I)$.



\begin{theorem}
\label{bigtheoremregularity}
Let $I \subset S$ be a bihomogeneous ideal.
Then, there is $0 \leq s \leq \cd_{\mx}(I) - 1$, such that $\reg(I) + (s,0) \subset \xreg(I)$.
\end{theorem}

\begin{proof}
    If $(a,b) \in \reg(I)$, then we have $H_{\mathfrak{b}}^i(I)_{(a',b')} = 0$ for all $i \geq 1$ and $(a',b') \geq (a - \lambda_x,b - \lambda_y)$ with $(\lambda_x,\lambda_y) \in \mathbb{Z}^2_{\geq 0}$ such that $\lambda_x + \lambda_y = i - 1$. In particular, $H_{\mathfrak{b}}^i(I)_{(a',b')} = 0$  for all $(a',b') \geq (a,b)$. This implies that $(a,b) \notin \supp_{\mathbb{Z}^2}(H^{\bullet}_{\mathfrak{b}}(I))^{\star}$. Hence, by Theorem \ref{def:chardinholanda}, we get $(a,b) \notin \supp_{\mathbb{Z}^2}(H^{\bullet}_{\mx}(I))^{\star}$. It follows that $$H_{\mx}^i(I)_{(a',b')} = 0 \text{ for all }(a',b') \geq (a,b) \text{ and }i \geq 1.$$ Therefore, there is some $0 \leq s \leq \cd_{\mx}(I) - 1$ such that for $i \geq 1$, we have $H_{\mx}^i(I)_{(a',b')}$ $ = 0$ for all $(a',b') \geq (a + s - (i-1), b)$. This implies that $(a + s, b) \in \xreg(I)$ and so does every $(a',b') \geq (a + s, b)$.
\end{proof}

\begin{remark}
    Remark \ref{remarkcohomological} implies that for every ideal, the integer $s$ appearing in the above theorem is bounded by $n$. In many cases, we can also bound the cohomological dimension using the dimension of $I$, as a module over $\mathbf{k}[x_0,\dots,x_n]$; see \cite[Proposition 2.14]{grothendieckvanishing}. 
\end{remark}

As a consequence of Theorem~\ref{bigtheoremregularity}, we derive a relation between $\reg(I)$ and the minimal generators of $\bigin(I)$.

\begin{corollary}
\label{theoremreg}
Let $I \subset S$ be a bihomogeneous ideal and $(a,b) \in \reg(I) \cap \mathbb{Z}^2_{\succsim 0}$.
Then, there is $1 \leq s \leq \cd_{\mx}(I)$ such that for every $(a',b') \geq (a + s, b)$, there is no minimal generator of $\bigin(I)$ of bidegree $(a',b')$.
\end{corollary}
\begin{proof}
    The proof follows from applying Theorem \ref{theoremxreg} and Theorem \ref{bigtheoremregularity}.
\end{proof}




Using Theorem \ref{chardinholandator}, we can also relate $\xreg(I)$ with the Betti numbers of $I$.

 \begin{theorem}
 \label{theorembetti}
Let $I \subset S$ be any bihomogeneous ideal and let $(a,b) \in \xreg(I) \cap \mathbb{Z}^2_{\succsim 0}$.
 Then, 
    $(a + n + 1,b + m + 1) \notin \beta_{i}(I)$ for all $i \geq 1$ .
\end{theorem}

\begin{proof}
    If $(a,b) \in \xreg(I)$, then $(a - i + 1,b) \notin \supp_{\mathbb{Z}^2}(H^i_{\mx}(I))^{\star}$ for all $i \geq 1$. In particular, $(a,b) \notin \supp_{\mathbb{Z}^2}(H^{\bullet}_{\mx}(I))^{\star}$. Using Theorem \ref{chardinholandator}, we derive that 
    $(a + n + 1,b + m + 1) \notin \cup_{i}\beta_i(I)^{\star}$, concluding the proof.  
\end{proof}

\begin{corollary}
     \label{theorembettireg}
     Let $I \subset S$ be a bihomogeneous ideal and let $(a,b) \in \xreg(I) \cap \mathbb{Z}^2_{\succsim 0}$.
     Then, there is $0 \leq s \leq \cd_{\mx}(I) - 1$, such that
    $\beta_{i,(a',b')} = 0$
    for all $i \geq 1$ and $(a',b') \geq (a + n + s + 1,b + m + 1)$.
\end{corollary}

\begin{proof}
    Apply Proposition \ref{theorembetti} and Theorem \ref{bigtheoremregularity}.
\end{proof}

   We refer to \cite[Corollary 3.8]{botbol2012castelnuovo} for a finer version of Corollary \ref{theorembettireg}.
\begin{example}
    We continue with Example~\ref{example2}.  In Figure \ref{fig:examplebetti}, we illustrate the region $\xreg(I) + (3,2)$ and the Betti numbers, i.e., the bidegrees $(a,b)$ such that there is $i \geq 1$ with $\beta_{i,(a,b)}(I) \neq 0$ in the minimal free resolution of $I$. Proposition~\ref{theorembetti} guarantees that there is no Betti number in the region $\xreg(I) + (3,2)$. \ccc{The shift in this example was calculated by bounding the cohomological dimension of $I$ with respect to $I$. This bound can be found by considering a Noether normalization of the $\mathbf{k}[x_0,\dots,x_n]$-module $(S/I)_{(*,0)}$. For more details, we refer to the third authors PhD thesis \cite[Section 5.4]{carlesthesis}. } 
     \begin{figure}
        \centering
\input{tikzfiles/fig10.tikz}
        \caption{We illustrate in \textcolor{olive}{olive} the region \textcolor{olive}{$\xreg(I) + (3,2)$} and in \textcolor{pink}{pink} squares, the Betti numbers, i.e. the bidegrees $(a,b)$ such that there is $i \geq 1$ with $\beta_{i,(a,b)}(I) \neq 0$. }
        \label{fig:examplebetti}
    \end{figure}
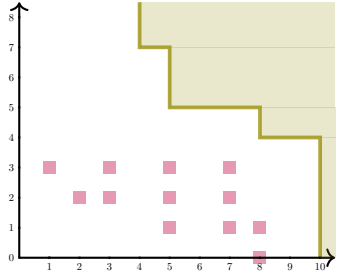
\end{example}

\section{Final remarks and open questions} \label{sec:final}

The definition of the partial regularity region $\xreg(I)$ allowed us to (\ccc{partially}) generalize the criterion of Bayer and Stillman in \cite{bayer_criterion_1987} to the case of the bigeneric initial ideal $\bigin(I)$, using the $\DRL$ monomial order Eq.~\eqref{eqmono}. This region relies solely on the local cohomology modules of $I$ with respect to $\mx$ and not on the properties of $\bigin(I)$. However, as we observed in Figure \ref{fig:examplebrucexreg}, there are unbounded regions that do not intersect $\xreg(I)$. Therefore, providing a finer bounding region to describe the bidegrees of $\bigin(I)$ remains an open problem. 

\ccc{The following result characterizes the bidegrees of the generators of $\bigin(I)$ in terms of properties of colon ideals with respect to $\bigin(I)$. Its proof is very similar to the proof of Theorem \ref{theoremexistenceelementsxreg}, but we include it for completeness.}


\begin{proposition}
\label{openproblem}
  Let $I \subset S$ be a bihomogeneous ideal and $(a,b) \in \mathbb{Z}^2_{\succsim 0}$. Then, there is no minimal generator of $\bigin(I)$ of bidegree $(a,b)$, if and only if,
  \begin{equation}
    \label{openproblem_eq}
      (J_{k-1}:x_k)_{(a - 1,b)} = (J_{k-1} + \my(J_{k-1}:x_k))_{(a-1,b)} \ \ \textrm{ for all } k = 0,\dots,n,
\end{equation}
where $J_{k-1} = (\bigin(I),x_0,\dots,x_{k-1})$.
\end{proposition}

\ccc{
\begin{proof}
If there is $0 \leq k \leq n$ such that
$$(J_{k-1}:x_k)_{(a - 1,b)} \neq (J_{k-1} + \my(J_{k-1}:x_k))_{(a-1,b)},$$
then there is a monomial $x^{\alpha}y^{\beta} \notin (J_{k-1} + \my(J_{k-1}:x_k))_{(a-1,b)}$ such that $x_kx^{\alpha}y^{\beta} \in (\lb{J_{k-1}})_{(a,b)}$. If \lb{$x_kx^{\alpha}y^{\beta}$} is not a minimal generator of $\bigin(I)$, then we write it as $x_kx^{\alpha}y^{\beta} = \overline{z}z^{\gamma}$ for some minimal generator $z^{\gamma} \in \bigin(I)$. \lb{If} $\overline{z} \neq 1$, \lb{we distinguish four different cases}: 
\begin{itemize}
   \item[-] If there is $0 \leq k' \leq k-1$ such that $x_{k'}$ divides $\overline{z}$, we derive that $x^{\alpha}y^{\beta} \in (x_{k'}) \subset J_{k-1}$, leading to \lb{a} contradiction.
   \item[-] If $x_k$ divides $\overline{z}$, \lb{then} $x^{\alpha}y^{\beta} = \frac{\overline{z}}{x_k}z^{\gamma} \in \bigin(I) \subset J_{k-1}$, \lb{which gives again a contradiction}.
   \item[-] If there is $k' > k$ such that $x_{k'}$ divides $\overline{z}$, then, by the properties of $\bigin(I)$ we derive first that $z^{\gamma} = x_{k}z^{\gamma'}$ and then, \lb{using Lemma \ref{lemmabiborel}, that $x_{k'}z^{\gamma'} \in \bigin(I)$}. \lb{It follows that}  $x^{\alpha}y^{\beta} = x_{k'}\frac{\overline{z}}{x_{k'}}z^{\gamma'} \in \bigin(I)$, \lb{again a contradiction}.
   \item[-] \lb{Finally,} if $y_{k'}$ divides $\overline{z}$ for some $0 \leq k' \leq m$, \lb{then} $x_kx^{\alpha}y^{\beta} \in \mathfrak{m}_y(J_{k-1}:x_{k})$, \lb{which gives again a contradiction}.
\end{itemize}
\lb{Therefore, we deduce that} $x_kx^{\alpha}y^{\beta}$ must be a minimal generator of $\bigin(I)$, \lb{which proves that if there is no mimimal generator of $\bigin(I)$ of bidegree $(a,b)$ then the conditions \eqref{openproblem_eq} hold}. \lb{The converse implication can be derived with similar arguments; we leave the details to the reader.}
\end{proof}
}
 The main remaining challenge is to be able to characterize the left hand side of Eq.~\eqref{openproblem_eq} in terms of the algebraic properties of $I$ (local cohomology, Betti numbers...) by providing a tight bounding region for the description of the bidegrees of the generators of $\bigin(I)$.

\medskip

\ccc{We \lb{expect that} the results we obtained in the bihomogeneous setting will also  hold in the multihomogeneous case, i.e. over $\mathbb{P}^{n_1} \times \dots \times \mathbb{P}^{n_r}$ and using the multihomogeneous analogue of the bigeneric initial ideal ($  \multigin(I)$). \lb{Indeed,} in each step of our proofs we only used $i)$ the local cohomology with respect to one group of variables, $ii)$ the monomial order in Eq.~\eqref{eqmono}, for which this group of variables is the smallest and $iii)$ the results in \cite{chardinholanda} which are \lb{valid} in the multihomogeneous case. However, as $\xreg(I)$ only depends on one group of variables, the \textit{looseness} of the corresponding bounding region for the multidegrees of the minimal generators of $\multigin(I)$ would increase with a higher number of factors involved. Therefore, further work is required in order to provide sharp bounding regions in the multihomogeneous \lb{setting} too.
}

\ccc{\lb{Finally,} another interesting open problem we can mention, related to complexity aspects, is \lb{the generalization of} the doubly exponential bounds known in the single graded case \cite{galligo74, caviglia_sbarra_2005} to the multigraded setting.}




\bibliographystyle{siamplain}
\bibliography{cm}

\end{document}

%% file: ex_shared.tex
\usepackage{lipsum}
\usepackage{amsfonts}
\usepackage{graphicx}
\usepackage{epstopdf}
\usepackage{algorithmic}
\usepackage{pgf}

\ifpdf
  \DeclareGraphicsExtensions{.eps,.png,.jpg}
\else
  \DeclareGraphicsExtensions{.eps}
\fi

\newsiamremark{remark}{Remark}
\newsiamremark{notation}{Notation}
\newsiamremark{hypothesis}{Hypothesis}
\crefname{hypothesis}{Hypothesis}{Hypotheses}
\newsiamthm{claim}{Claim}

\headers{}{}

\title{Bigraded Castelnuovo-Mumford Regularity and Gr\"obner bases}

\author{Mat\'ias Bender\thanks{Inria \& CMAP, CNRS, \'Ecole Polytechnique, Institut polytechnique de Paris, Palaiseau, France 
(\email{matias.bender@inria.fr}).}
\and Laurent Bus\'e\thanks{Universit\'e Cote d'Azur, Inria, Sophia Antipolis, France. 
  (\email{laurent.buse@inria.fr}).}
\and Carles Checa\thanks{National Kapodistrian University of Athens \& Athena RC, Greece
  (\email{ccheca@di.uoa.gr}).} \and Elias Tsigaridas\thanks{Inria Paris \& Institut de Math\'ematiques de Jussieu - Paris Rive Gauche, Sorbonne Universit\'e and Paris Universit\'e, France.
(\email{elias.tsigaridas@inria.fr})}
}

\usepackage{amsopn}


%% file: tikzfiles/fig9
\begin{tikzpicture}[thick,scale=0.5, every node/.style=transform shape]

\filldraw[green] (1,3) circle (3pt) node[anchor=north] {};
\filldraw[green] (2,2) circle (3pt) node[anchor=north] {};
\filldraw[green] (3,2) circle (3pt) node[anchor=north] {};
\filldraw[green] (5,1) circle (3pt) node[anchor=north] {};
\filldraw[green] (8,0) circle (3pt) node[anchor=north] {};

\filldraw[black] (6,1) circle (3pt) node[anchor=north] {};

\filldraw[black] (2,3) circle (3pt) node[anchor=north] {};

\filldraw[black] (2,4) circle (3pt) node[anchor=north] {};
\filldraw[black] (2,5) circle (3pt) node[anchor=north] {};
\filldraw[black] (2,6) circle (3pt) node[anchor=north] {};
\filldraw[black] (2,7) circle (3pt) node[anchor=north] {};
\filldraw[black] (2,8) circle (3pt) node[anchor=north] {};

\filldraw[red] (4,4) node[anchor=north] {};
\filldraw[red] (4,5) node[anchor=north] {};
\filldraw[red] (4,6) node[anchor=north] {};
\filldraw[black] (5,1) node[anchor=north] {\color{red} };
\filldraw[white] (5,4) circle (2pt)  node[anchor=north] {\color{red} };
\filldraw[red] (5,5) node[anchor=north] {};
\filldraw[red] (5,6) node[anchor=north] {};
\filldraw[black] (6,1)node[anchor=north] {\color{red} };
\filldraw[red] (6,2) node[anchor=north] {};
\filldraw[red] (6,3) node[anchor=north] {};
\filldraw[red] (6,4) node[anchor=north] {};
\filldraw[red] (6,5) node[anchor=north] {};

\draw[thick,->] (0,0) -- (9.5,0) node[anchor=north west] {};

\draw[line width = 0.5mm,opacity=0.75,color=brown,-] (9.5,0) -- (9,0.) -- (9,1) -- (7,1) -- (7,2.) -- (4,2) -- (4,4) -- (3,4.0) -- (3.05,8.51) node[anchor=south east] {};

\fill[color=brown, opacity=0.2] (3,4) rectangle (9.5,8.5) ;

\fill[color=brown, opacity=0.2] (4,2) rectangle (9.5,4) ;

\fill[color=brown, opacity=0.2] (7,1) rectangle (9.5,2) ;

\fill[color=brown, opacity=0.2] (9,0) rectangle (9.5,1) ;

\draw[thick,->] (0,0) -- (0,8.5) node[anchor=south east] {};
\foreach \x in {,1,2,3,4,5,6,7,8,9}
 \draw (\x,1pt) -- (\x,-1pt) node[anchor=north] {$\x$};

 \draw[thick,color=cyan,-] (2.5,8.5) -- (2.5,1.5) -- (1.5,1.5) -- (1.5,8.5) node[anchor=south east] {};

  \draw[thick,color=cyan,-] (2.5,3.5) -- (2.5,1.5) -- (3.5,1.5) -- (3.5,3.5) -- (2.5,3.5) node[anchor=south east] {};

 \draw[thick,color=cyan,-] (5.5,0.5) -- (5.5,1.5) -- (6.5,1.5) -- (6.5,0.5) -- (5.5,0.5) node[anchor=south east] {};

  \draw[thick,color=cyan,-] (7.5,0.5) -- (7.5,-0.5) -- (8.5,-0.5) -- (8.5,0.5) -- (7.5,0.5) node[anchor=south east] {};

\foreach \y in {0,1,2,3,4,5,6,7,8}
   \draw (1pt,\y cm) -- (-1pt,\y cm) node[anchor=east] {$\y$};

  \draw[thick,color=purple,-] (7.5,8.5) -- (7.5,-0.5) -- (8.5,-0.5) -- (8.5,8.5) node[anchor=south east] {};



\end{tikzpicture}

%% file: tikzfiles/fig1-2.tikz
  \begin{tikzpicture}[thick,scale=0.5, every node/.style={transform shape}]

\filldraw[green] (1,3) circle (3pt) node[anchor=north] {};
\filldraw[green] (2,2) circle (3pt) node[anchor=north] {};
\filldraw[green] (3,2) circle (3pt) node[anchor=north] {};
\filldraw[green] (5,1) circle (3pt) node[anchor=north] {};
\filldraw[green] (8,0) circle (3pt) node[anchor=north] {};

\filldraw[black] (6,1) circle (3pt) node[anchor=north] {};

\filldraw[black] (2,3) circle (3pt) node[anchor=north] {};

\filldraw[black] (2,4) circle (3pt) node[anchor=north] {};
\filldraw[black] (2,5) circle (3pt) node[anchor=north] {};
\filldraw[black] (2,6) circle (3pt) node[anchor=north] {};
\filldraw[black] (2,7) circle (3pt) node[anchor=north] {};
\filldraw[black] (2,8) circle (3pt) node[anchor=north] {};

\filldraw[red] (4,4) node[anchor=north] {};
\filldraw[red] (4,5) node[anchor=north] {};
\filldraw[red] (4,6) node[anchor=north] {};
\filldraw[black] (5,1) node[anchor=north] {\color{red} };
\filldraw[white] (5,3) circle (2pt)  node[anchor=north] {\color{red} };
\filldraw[red] (5,5) node[anchor=north] {};
\filldraw[red] (5,6) node[anchor=north] {};
\filldraw[black] (6,1)node[anchor=north] {\color{red} };
\filldraw[red] (6,2) node[anchor=north] {};
\filldraw[red] (6,3) node[anchor=north] {};
\filldraw[red] (6,4) node[anchor=north] {};
\filldraw[red] (6,5) node[anchor=north] {};

\draw[thick,->] (0,0) -- (8.5,0) node[anchor=north west] {};

\draw[thick,->] (0,0) -- (0,8.5) node[anchor=south east] {};
\foreach \x in {,1,2,3,4,5,6,7,8}
 \draw (\x,1pt) -- (\x,-1pt) node[anchor=north] {$\x$};

\foreach \y in {0,1,2,3,4,5,6,7,8}
   \draw (1pt,\y cm) -- (-1pt,\y cm) node[anchor=east] {$\y$};

\end{tikzpicture}

%% file: tikzfiles/fig1-1.tikz
\begin{tikzpicture}[thick,scale=0.5, every node/.style={transform shape}]

\filldraw[green] (1,3) circle (3pt) node[anchor=north] {};
\filldraw[green] (2,2) circle (3pt) node[anchor=north] {};
\filldraw[green] (3,2) circle (3pt) node[anchor=north] {};
\filldraw[green] (5,1) circle (3pt) node[anchor=north] {};
\filldraw[green] (8,0) circle (3pt) node[anchor=north] {};

\filldraw[black] (2,3) circle (3pt) node[anchor=north] {};
\filldraw[black] (5,2) circle (3pt) node[anchor=north] {};
\filldraw[black] (4,2) circle (3pt) node[anchor=north] {};

\filldraw[black] (7,1) circle (3pt) node[anchor=north] {};
\filldraw[black] (6,1) circle (3pt) node[anchor=north] {};

\filldraw[red] (4,4) node[anchor=north] {};
\filldraw[red] (4,5) node[anchor=north] {};
\filldraw[red] (4,6) node[anchor=north] {};
\filldraw[black] (5,1) node[anchor=north] {\color{red} };
\filldraw[white] (5,4) circle (2pt)  node[anchor=north] {\color{red} };
\filldraw[red] (5,5) node[anchor=north] {};
\filldraw[red] (5,6) node[anchor=north] {};
\filldraw[black] (6,1)node[anchor=north] {\color{red} };
\filldraw[red] (6,2) node[anchor=north] {};
\filldraw[red] (6,3) node[anchor=north] {};
\filldraw[red] (6,4) node[anchor=north] {};
\filldraw[red] (6,5) node[anchor=north] {};

\draw[thick,->] (0,0) -- (8.5,0) node[anchor=north west] {};

\draw[thick,->] (0,0) -- (0,8.5) node[anchor=south east] {};
\foreach \x in {,1,2,3,4,5,6,7,8}
 \draw (\x,1pt) -- (\x,-1pt) node[anchor=north] {$\x$};

\foreach \y in {0,1,2,3,4,5,6,7,8}
   \draw (1pt,\y cm) -- (-1pt,\y cm) node[anchor=east] {$\y$};

\end{tikzpicture}

%% file: tikzfiles/fig3.tikz
\begin{tikzpicture}[thick,scale=0.5, every node/.style={transform shape}]

\filldraw[green] (1,3) circle (3pt) node[anchor=north] {};
\filldraw[green] (2,2) circle (3pt) node[anchor=north] {};
\filldraw[green] (3,2) circle (3pt) node[anchor=north] {};
\filldraw[green] (5,1) circle (3pt) node[anchor=north] {};
\filldraw[green] (8,0) circle (3pt) node[anchor=north] {};

\filldraw[black] (6,1) circle (3pt) node[anchor=north] {};

\filldraw[black] (2,3) circle (3pt) node[anchor=north] {};

\filldraw[black] (2,4) circle (3pt) node[anchor=north] {};
\filldraw[black] (2,5) circle (3pt) node[anchor=north] {};
\filldraw[black] (2,6) circle (3pt) node[anchor=north] {};
\filldraw[black] (2,7) circle (3pt) node[anchor=north] {};
\filldraw[black] (2,8) circle (3pt) node[anchor=north] {};

\filldraw[purple] (4,9.3) node[anchor=north] {$\mathfrak{R}_y(\bigin(I))$};

\filldraw[blue] (4,4.3) node[anchor=north] {$\mathfrak{R}_y(I)$};

\filldraw[purple] (9.2,2) node[anchor=north] {$\mathfrak{R}_x(I)$};

  \draw[thick,color=purple,-] (7.5,8.5) -- (7.5,-0.5) -- (8.5,-0.5) -- (8.5,8.5) node[anchor=south east] {};

  \draw[thick,color=blue,-] (9.5,2.5) -- (-0.5,2.5) -- (-0.5,3.5) -- (9.5,3.5) node[anchor=south east] {};

  \draw[thick,color=purple,-] (9.5,7.5) -- (-0.5,7.5) -- (-0.5,8.5) -- (9.5,8.5) node[anchor=south east] {};

\filldraw[red] (4,4) node[anchor=north] {};
\filldraw[red] (4,5) node[anchor=north] {};
\filldraw[red] (4,6) node[anchor=north] {};
\filldraw[black] (5,1) node[anchor=north] {\color{red} };
\filldraw[white] (5,4) circle (2pt)  node[anchor=north] {\color{red} };
\filldraw[red] (5,5) node[anchor=north] {};
\filldraw[red] (5,6) node[anchor=north] {};
\filldraw[black] (6,1)node[anchor=north] {\color{red} };
\filldraw[red] (6,2) node[anchor=north] {};
\filldraw[red] (6,3) node[anchor=north] {};
\filldraw[red] (6,4) node[anchor=north] {};
\filldraw[red] (6,5) node[anchor=north] {};

\draw[thick,->] (0,0) -- (9.5,0) node[anchor=north west] {};

\draw[thick,->] (0,0) -- (0,8.5) node[anchor=south east] {};
\foreach \x in {,1,2,3,4,5,6,7,8,9}
 \draw (\x,1pt) -- (\x,-1pt) node[anchor=north] {$\x$};

\foreach \y in {0,1,2,3,4,5,6,7,8}
   \draw (1pt,\y cm) -- (-1pt,\y cm) node[anchor=east] {$\y$};

\end{tikzpicture}

%% file: tikzfiles/fig4.tikz
\begin{tikzpicture}[thick,scale=0.65, every node/.style={transform shape}]

\filldraw[red] (1,1) node[anchor=north] {};
\filldraw[red] (1,2) node[anchor=north] {};
\filldraw[green] (1,3) circle (3pt) node[anchor=north] {\color{red} };
\filldraw[black] (1,4) circle (3pt) node[anchor=north] {};
\filldraw[black] (1,5) circle (3pt) node[anchor=north] {};
\filldraw[black] (1,6) circle (3pt) node[anchor=north] {};

\filldraw[red] (2,1) node[anchor=north] {};
\filldraw[red] (2,2) node[anchor=north] {};
\filldraw[black] (2,3)   node[anchor=north] {\color{red} };
\filldraw[black] (2,4)   node[anchor=north] {\color{red} };
\filldraw[black] (2,5)  node[anchor=north] {\color{red} };
\filldraw[red] (2,6) node[anchor=north] {};

\filldraw[green] (3,1) circle (3pt) node[anchor=north] {\color{red} };
\filldraw[black] (3,2) circle (3pt) node[anchor=north] {};
\filldraw[black] (3,3) circle (3pt)  node[anchor=north] {\color{red} };
\filldraw[black] (3,4) circle (3pt)  node[anchor=north] {\color{red} };
\filldraw[black] (3,5) circle (3pt)  node[anchor=north] {\color{red} };
\filldraw[white] (3,5) circle (2pt)  node[anchor=north] {\color{red} };
\filldraw[red] (3,6) node[anchor=north] {};

\filldraw[black] (4,1)  node[anchor=north] {\color{red} };
\filldraw[black] (4,2) circle (3pt) node[anchor=north] {};
\filldraw[black] (4,3) circle (3pt)  node[anchor=north] {\color{red} };
\filldraw[red] (4,4) node[anchor=north] {};
\filldraw[red] (4,5) node[anchor=north] {};
\filldraw[red] (4,6) node[anchor=north] {};

\filldraw[black] (5,1) node[anchor=north] {\color{red} };
\filldraw[black] (5,2) circle (3pt) node[anchor=north] {};
\filldraw[black] (5,3) circle (3pt)  node[anchor=north] {\color{red} };
\filldraw[white] (5,3) circle (2pt)  node[anchor=north] {\color{red} };
\filldraw[red] (5,4) node[anchor=north] {};
\filldraw[red] (5,5) node[anchor=north] {};
\filldraw[red] (5,6) node[anchor=north] {};

\filldraw[black] (6,1)node[anchor=north] {\color{red} };
\filldraw[red] (6,2) node[anchor=north] {};
\filldraw[red] (6,3) node[anchor=north] {};
\filldraw[red] (6,4) node[anchor=north] {};
\filldraw[red] (6,5) node[anchor=north] {};
\filldraw[red] (6,6) node[anchor=north] {};

\draw[thick,->] (0,0) -- (6.5,0) node[anchor=north west] {};

\filldraw[red] (5,6) node[anchor=north] {$\reg(I)$};

\draw[thick,->] (0,0) -- (0,6.5) node[anchor=south east] {};
\foreach \x in {,1,2,3,4,5,6}
 \draw (\x,1pt) -- (\x,-1pt) node[anchor=north] {$\x$};
\foreach \y in {0,1,2,3,4,5,6}
    \draw (1pt,\y cm) -- (-1pt,\y cm) node[anchor=east] {$\y$};

\fill[color=red, opacity=0.2] (3,5) rectangle (5,6.5) ;

\fill[color=red, opacity=0.2] (5,3) rectangle (6.5,5) ;

\fill[color=red, opacity=0.2] (5,5) rectangle (6.5,6.5) ;

\draw[line width = 0.5mm,color=red,-] (3,6.5) -- (3,5) -- (5,5) -- (5,3) -- (6.5,3) node[anchor=south east] {};

\draw[line width = 0.5mm,color=brown,-] (6.5,2.5) -- (0.4,2.5) -- (0.4,1.5) -- (6.5,1.5) node[anchor=south east] {};

\draw[line width = 0.5mm,color=blue,-] (0.5,6.5) -- (0.5,0.5) -- (1.5,0.5) -- (1.5,6.5) node[anchor=south east] {};

\end{tikzpicture}

%% file: tikzfiles/fig5-2.tikz
\begin{tikzpicture}[thick,scale=0.5, every node/.style={transform shape}]

\filldraw[green] (1,3) circle (3pt) node[anchor=north] {};
\filldraw[green] (2,2) circle (3pt) node[anchor=north] {};
\filldraw[green] (3,2) circle (3pt) node[anchor=north] {};
\filldraw[green] (5,1) circle (3pt) node[anchor=north] {};
\filldraw[green] (8,0) circle (3pt) node[anchor=north] {};

\filldraw[black] (6,1) circle (3pt) node[anchor=north] {};

\filldraw[black] (2,3) circle (3pt) node[anchor=north] {};

\filldraw[black] (2,4) circle (3pt) node[anchor=north] {};
\filldraw[black] (2,5) circle (3pt) node[anchor=north] {};
\filldraw[black] (2,6) circle (3pt) node[anchor=north] {};
\filldraw[black] (2,7) circle (3pt) node[anchor=north] {};
\filldraw[black] (2,8) circle (3pt) node[anchor=north] {};

\filldraw[red] (4,4) node[anchor=north] {};
\filldraw[red] (4,5) node[anchor=north] {};
\filldraw[red] (4,6) node[anchor=north] {};
\filldraw[black] (5,1) node[anchor=north] {\color{red} };
\filldraw[white] (5,3) circle (2pt)  node[anchor=north] {\color{red} };
\filldraw[red] (5,5) node[anchor=north] {};
\filldraw[red] (5,6) node[anchor=north] {};
\filldraw[black] (6,1)node[anchor=north] {\color{red} };
\filldraw[red] (6,2) node[anchor=north] {};
\filldraw[red] (6,3) node[anchor=north] {};
\filldraw[red] (6,4) node[anchor=north] {};
\filldraw[red] (6,5) node[anchor=north] {};

\draw[thick,->] (0,0) -- (8.5,0) node[anchor=north west] {};



\draw[line width = 0.5mm,opacity=0.75,color=green,-]  (8.5,2) -- (5,2) -- (5,3) -- (3,3) -- (3,4) -- (2,4) -- (2,8.5) node[anchor=south east] {};

\draw[line width = 0.5mm,opacity=0.75,color=purple,-]  (8.5,8) -- (2,8) -- (2,8.5) node[anchor=south east] {};

\draw[thick,->] (0,0) -- (0,8.5) node[anchor=south east] {};
\foreach \x in {,1,2,3,4,5,6,7,8}
 \draw (\x,1pt) -- (\x,-1pt) node[anchor=north] {$\x$};

 \fill[color=purple, opacity=0.3] (2,8) rectangle (8.5,8.5) ;

 \fill[color=green, opacity=0.2] (3,3) rectangle (8.5,4) ;

 \fill[color=green, opacity=0.2] (2,4) rectangle (8.5,8.5) ;

 \fill[color=green, opacity=0.2] (5,2) rectangle (8.5,3) ;




\foreach \y in {0,1,2,3,4,5,6,7,8}
   \draw (1pt,\y cm) -- (-1pt,\y cm) node[anchor=east] {$\y$};

\end{tikzpicture}

%% file: tikzfiles/fig6.tikz
\begin{tikzpicture}[thick,scale=0.45, every node/.style={transform shape}]

\filldraw[green] (1,3) circle (3pt) node[anchor=north] {};
\filldraw[green] (2,2) circle (3pt) node[anchor=north] {};
\filldraw[green] (3,2) circle (3pt) node[anchor=north] {};
\filldraw[green] (5,1) circle (3pt) node[anchor=north] {};
\filldraw[green] (8,0) circle (3pt) node[anchor=north] {};

\filldraw[black] (6,1) circle (3pt) node[anchor=north] {};

\filldraw[black] (2,3) circle (3pt) node[anchor=north] {};

\filldraw[black] (2,4) circle (3pt) node[anchor=north] {};
\filldraw[black] (2,5) circle (3pt) node[anchor=north] {};
\filldraw[black] (2,6) circle (3pt) node[anchor=north] {};
\filldraw[black] (2,7) circle (3pt) node[anchor=north] {};
\filldraw[black] (2,8) circle (3pt) node[anchor=north] {};

\filldraw[red] (4,4) node[anchor=north] {};
\filldraw[red] (4,5) node[anchor=north] {};
\filldraw[red] (4,6) node[anchor=north] {};
\filldraw[black] (5,1) node[anchor=north] {\color{red} };
\filldraw[white] (5,4) circle (2pt)  node[anchor=north] {\color{red} };
\filldraw[red] (5,5) node[anchor=north] {};
\filldraw[red] (5,6) node[anchor=north] {};
\filldraw[black] (6,1)node[anchor=north] {\color{red} };
\filldraw[red] (6,2) node[anchor=north] {};
\filldraw[red] (6,3) node[anchor=north] {};
\filldraw[red] (6,4) node[anchor=north] {};
\filldraw[red] (6,5) node[anchor=north] {};

\draw[line width = 0.5mm,opacity=0.75,color=olive,-] (0,9.5) -- (0,3) -- (5,3) -- (5,2) -- (9.5,2) node[anchor=south east] {};

\draw[line width = 0.5mm,opacity=0.75,color=blue,-] (0,9.5) -- (0,8) -- (9.5,8) node[anchor=south east] {};

\fill[color=blue, opacity=0.2] (0,8) rectangle (9.5,9.5) ;

\fill[color=olive, opacity=0.2] (0,3) rectangle (9.5,9.5) ;

\fill[color=olive, opacity=0.2] (5,2) rectangle (9.5,3) ;

\foreach \y in {0,1,2,3,4,5,6,7,8,9}
   \draw (1pt,\y cm) -- (-1pt,\y cm) node[anchor=east] {$\y$};

\draw[thick,->] (0,0) -- (9.5,0) node[anchor=south east] {};

\draw[thick,->] (0,0) -- (0,9.5) node[anchor=south east] {};
\foreach \x in {,1,2,3,4,5,6,7,8,9}
 \draw (\x,1pt) -- (\x,-1pt) node[anchor=north] {$\x$};

\end{tikzpicture}

%% file: tikzfiles/fig7.tikz
 \beginpgfgraphicnamed{levels}
\begin{tikzpicture}[thick,scale=0.5]

\filldraw[green] (1,3) circle (3pt) node[anchor=north] {};
\filldraw[green] (2,2) circle (3pt) node[anchor=north] {};
\filldraw[green] (3,2) circle (3pt) node[anchor=north] {};
\filldraw[green] (5,1) circle (3pt) node[anchor=north] {};
\filldraw[green] (8,0) circle (3pt) node[anchor=north] {};

\filldraw[black] (6,1) circle (3pt) node[anchor=north] {};

\filldraw[black] (2,3) circle (3pt) node[anchor=north] {};

\filldraw[black] (2,4) circle (3pt) node[anchor=north] {};
\filldraw[black] (2,5) circle (3pt) node[anchor=north] {};
\filldraw[black] (2,6) circle (3pt) node[anchor=north] {};
\filldraw[black] (2,7) circle (3pt) node[anchor=north] {};
\filldraw[black] (2,8) circle (3pt) node[anchor=north] {};

\filldraw[red] (4,4) node[anchor=north] {};
\filldraw[red] (4,5) node[anchor=north] {};
\filldraw[red] (4,6) node[anchor=north] {};
\filldraw[black] (5,1) node[anchor=north] {\color{red} };
\filldraw[white] (5,4) circle (2pt)  node[anchor=north] {\color{red} };
\filldraw[red] (5,5) node[anchor=north] {};
\filldraw[red] (5,6) node[anchor=north] {};
\filldraw[black] (6,1)node[anchor=north] {\color{red} };
\filldraw[red] (6,2) node[anchor=north] {};
\filldraw[red] (6,3) node[anchor=north] {};
\filldraw[red] (6,4) node[anchor=north] {};
\filldraw[red] (6,5) node[anchor=north] {};

\draw[thick,->] (0,0) -- (9.5,0) node[anchor=north west] {};

\draw[line width = 0.5mm,opacity=0.75,color=brown,-] (9.5,0) -- (9,0.) -- (9,1) -- (7,1) -- (7,2.) -- (4,2) -- (4,4) -- (3,4.0) -- (3.05,8.51) node[anchor=south east] {};

\fill[color=brown, opacity=0.2] (3,4) rectangle (9.5,8.5) ;

\fill[color=brown, opacity=0.2] (4,2) rectangle (9.5,4) ;

\fill[color=brown, opacity=0.2] (7,1) rectangle (9.5,2) ;

\fill[color=brown, opacity=0.2] (9,0) rectangle (9.5,1) ;

\draw[thick,->] (0,0) -- (0,8.5) node[anchor=south east] {};
\foreach \x in {,1,2,3,4,5,6,7,8,9}
 \draw (\x,1pt) -- (\x,-1pt) node[anchor=north] {$\x$};

 \draw[thick,color=cyan,-] (2.5,8.5) -- (2.5,1.5) -- (1.5,1.5) -- (1.5,8.5) node[anchor=south east] {};

  \draw[thick,color=cyan,-] (2.5,3.5) -- (2.5,1.5) -- (3.5,1.5) -- (3.5,3.5) -- (2.5,3.5) node[anchor=south east] {};

 \draw[thick,color=cyan,-] (5.5,0.5) -- (5.5,1.5) -- (6.5,1.5) -- (6.5,0.5) -- (5.5,0.5) node[anchor=south east] {};

  \draw[thick,color=cyan,-] (7.5,0.5) -- (7.5,-0.5) -- (8.5,-0.5) -- (8.5,0.5) -- (7.5,0.5) node[anchor=south east] {};

\foreach \y in {0,1,2,3,4,5,6,7,8}
   \draw (1pt,\y cm) -- (-1pt,\y cm) node[anchor=east] {$\y$};

\end{tikzpicture}
\endpgfgraphicnamed

%% file: tikzfiles/fig10.tikz
\begin{tikzpicture}[thick,scale=0.4, every node/.style={transform shape}]

\fill[purple!40!white] (0.8,2.8) rectangle (1.2,3.2);

\fill[purple!40!white] (1.8,1.8) rectangle (2.2,2.2);

\fill[purple!40!white] (2.8,1.8) rectangle (3.2,2.2);

\fill[purple!40!white] (4.8,0.8) rectangle (5.2,1.2);

\fill[purple!40!white] (7.8,-0.2) rectangle (8.2,0.2);

\fill[purple!40!white] (4.8,1.8) rectangle (5.2,2.2);

\fill[purple!40!white] (6.8,0.8) rectangle (7.2,1.2);

\fill[purple!40!white] (7.8,0.8) rectangle (8.2,1.2);

\fill[purple!40!white] (4.8,2.8) rectangle (5.2,3.2);

\fill[purple!40!white] (2.8,2.8) rectangle (3.2,3.2);

\fill[purple!40!white] (6.8,1.8) rectangle (7.2,2.2);

\fill[purple!40!white] (7.8,0.8) rectangle (8.2,1.2);

\fill[purple!40!white] (6.8,2.8) rectangle (7.2,3.2);

\filldraw[red] (4,4) node[anchor=north] {};
\filldraw[red] (4,5) node[anchor=north] {};
\filldraw[red] (4,6) node[anchor=north] {};
\filldraw[black] (5,1) node[anchor=north] {\color{red} };
\filldraw[white] (5,4) circle (2pt)  node[anchor=north] {\color{red} };
\filldraw[red] (5,5) node[anchor=north] {};
\filldraw[red] (5,6) node[anchor=north] {};
\filldraw[black] (6,1)node[anchor=north] {\color{red} };
\filldraw[red] (6,2) node[anchor=north] {};
\filldraw[red] (6,3) node[anchor=north] {};
\filldraw[red] (6,4) node[anchor=north] {};
\filldraw[red] (6,5) node[anchor=north] {};

\draw[line width = 0.5mm,opacity=0.75,color=olive,-] (10,0) -- (10,4) -- (8,4) -- (8,5.) -- (5,5) -- (5,7) -- (4,7.0) -- (4.,8.5) node[anchor=south east] {};

\draw[thick,->] (0,0) -- (0,8.5) node[anchor=south east] {};
\foreach \x in {,1,2,3,4,5,6,7,8,9,10}
 \draw (\x,1pt) -- (\x,-1pt) node[anchor=north] {$\x$};

\draw[thick,->] (0,0) -- (10.5,0) node[anchor=south east] {};
 
\foreach \y in {0,1,2,3,4,5,6,7,8}
   \draw (1pt,\y cm) -- (-1pt,\y cm) node[anchor=east] {$\y$};

\fill[color=olive, opacity=0.2]  (4,7) rectangle (10.5,8.5);

\fill[color=olive, opacity=0.2]  (5,5) rectangle (10.5,7);

\fill[color=olive, opacity=0.2]  (8,4) rectangle (10.5,5);

\fill[color=olive, opacity=0.2]  (10,0) rectangle (10.5,4);

\end{tikzpicture}